\documentclass[11pt,a4paper,twoside,final]{scrartcl}
\pdfoutput=1
\usepackage{a4wide}
\usepackage{amsfonts}
\usepackage{amsmath}
\usepackage{mathtools}
\usepackage{amssymb}
\usepackage{array}
\usepackage{amsthm}
\usepackage{dsfont}
\usepackage[utf8]{inputenc}
\usepackage[T1]{fontenc}
\usepackage{caption}
\usepackage{subcaption}
\usepackage{enumerate} 
\usepackage[
  backend=biber,
  style=numeric,
  useprefix=true,
  maxalphanames=4,
  bibencoding=utf8,
  giveninits=true,
  maxbibnames=9,
  date=year,
  doi=true,isbn=false,url=true,eprint=true]{biblatex}
\DeclareSourcemap{
  \maps[datatype=bibtex]{
    \map[overwrite]{
      \step[fieldsource=doi, final]
      \step[fieldset=url, null]
      \step[fieldset=eprint, null]
    }
  }
}
\usepackage{xifthen}
\usepackage[urlcolor=blue,colorlinks=false]{hyperref}
\usepackage[ruled,vlined]{algorithm2e}
\usepackage{graphicx}
\usepackage{color}
\usepackage{pgf}

\usepackage{multirow}
\usepackage{multicol}
\usepackage{tabularx}
\usepackage{diagbox}
\usepackage{longtable}
\usepackage{listings} 
\usepackage{cleveref}
\usepackage{tikz}
\usetikzlibrary{math}
\usetikzlibrary{patterns}
\usepackage{graphicx}
\usepackage{wrapfig}
\usepackage{changes} 
\usepackage[export]{adjustbox}

\newcommand{\N}{\ensuremath{\mathbb{N}}}

\newcommand{\T}{\ensuremath{\mathbb{T}}}

\newcommand{\Z}{\ensuremath{\mathbb{Z}}}

\newcommand{\R}{\ensuremath{\mathbb{R}}}
\newcommand{\C}{\ensuremath{\mathbb{C}}}

\newcommand{\ii}{\mathit{i}}
\newcommand{\e}{\textnormal{e}}

\newcommand{\eip}[1]{\textnormal{e}^{2\pi\ii{#1}}}
\newcommand{\eim}[1]{\textnormal{e}^{-2\pi\ii{#1}}}
\newcommand{\norm}[2][]{\ifthenelse{\isempty{#1}}%
  {\left\Vert #2\right\Vert}%
  {\left\Vert #2\right\Vert_{#1}}}

\newcommand{\normf}[1]{\left\Vert #1\right\Vert_\textnormal{F}}
\newcommand{\normlonetm}[1]{\left\Vert #1 \right\Vert_{L^1}}
\newcommand{\normltm}[1]{\left\Vert #1 \right\Vert_{L^2}}
\newcommand{\normlinftm}[1]{\left\Vert #1 \right\Vert_{L^\infty}}

\newcommand{\setcond}[2]{\left\{#1 : #2\right\}}

\newcommand{\abs}[1]{\left| #1 \right |}

\newcommand{\dirich}[1]{\ifthenelse{\isempty{#1}}
	{d_n}
	{d_n\left(#1\right)} }		
\newcommand{\dirichd}[1]{\ifthenelse{\isempty{#1}}
	{d_n'}
	{d_n'\left(#1\right)} }	
\newcommand{\dirichdd}[1]{\ifthenelse{\isempty{#1}}
	{d_n''}
	{d_n''\left(#1\right)} }	
\newcommand{\dirichm}[1]{\ifthenelse{\isempty{#1}}
	{\tilde d_n}
	{\tilde d_n\left(#1\right)} }

\newcommand{\diff}{\mathrm{d}}
\newcommand{\MM}{\mathcal{M}}

\newcommand{\de}{\delta}

\newcommand{\la}{\lambda}

\renewcommand{\mathbf}[1]{\ensuremath{\boldsymbol{#1}}}

\newcommand{\rank}{\operatorname{rank}}

\newcommand{\supp}{\operatorname{supp}}
\newcommand{\diag}{\operatorname{diag}}

\newcommand{\adj}{^{*}}

\newcommand{\idmat}[1]{\mathrm{I}_{#1}}
\newcommand{\scalarprod}[2]{\left\langle #1,#2 \right\rangle}
\newcommand{\lr}[1]{\left(#1\right)}

\newcommand{\eg}{\textit{e.g.\ }}

\DeclareMathOperator*{\smax}{\sigma_{\textnormal{max}}}
\DeclareMathOperator*{\smin}{\sigma_{\textnormal{min}}}

\DeclareMathOperator*{\argmin}{argmin}

\DeclareMathOperator*{\Lip}{Lip}

\renewcommand{\d}{\,\mathrm{d}}

\newcommand{\TV}[1]{\|{#1}\|_{\textnormal{TV}}}

\newtheorem{thm}{Theorem}[section]
\newtheorem{lemma}[thm]{Lemma}
\newtheorem{corollary}[thm]{Corollary}
\newtheorem{proposition}[thm]{Proposition}
\theoremstyle{definition}
\newtheorem{remark}[thm]{Remark}
\newtheorem{definition}[thm]{Definition}
\newtheorem{example}[thm]{Example}

\numberwithin{equation}{section}
\numberwithin{table}{section}
\numberwithin{figure}{section}

\newcommand{\bend}{\hspace*{0ex} \hfill \hbox{\vrule height
    1.5ex\vbox{\hrule width 1.4ex \vskip 1.4ex\hrule  width 1.4ex}\vrule
    height 1.5ex}}

\long\def\symbolfootnote[#1]#2{\begingroup%
\def\thefootnote{\fnsymbol{footnote}}\footnote[#1]{#2}\endgroup}

\crefname{lemma}{Lemma}{Lemmata}
\crefname{definition}{Definition}{Definitions}
\crefname{theorem}{Theorem}{Theorems}
\crefname{thm}{Theorem}{Theorems}
\crefname{corollary}{Corollary}{Corollaries}
\crefname{equation}{}{}
\crefname{remark}{Remark}{Remarks}
\crefname{algorithm}{Algorithm}{Algorithms}
\crefname{chapter}{Chapter}{Chapters}
\crefname{section}{Section}{Sections}
\crefname{table}{Table}{Tables}
\crefname{figure}{Figure}{Figures}
\crefname{example}{Example}{Examples}
\crefname{appendix}{Appendix}{Appendices}

\renewcommand{\thefootnote}{\fnsymbol{footnote}}

\allowdisplaybreaks
\title{Approximation and interpolation of singular measures by trigonometric polynomials}

\date{}

\author{Paul Catala\footnotemark[1] \and Mathias Hockmann\footnotemark[1] \and Stefan Kunis\footnotemark[1] \and Markus Wageringel\footnotemark[1]}
\newif\ifshow
\showtrue

\begin{document}
\maketitle

 \begin{abstract}
 
 Complex signed measures of finite total variation are a powerful signal model in many applications.
 Restricting to the $d$-dimensional torus, finitely supported measures allow for exact recovery if the trigonometric moments up to some order are known.
 Here, we consider the approximation of general measures, e.g., supported on a curve, by trigonometric polynomials of fixed degree with respect to the Wasserstein-1 distance.
 We prove sharp lower bounds for their best approximation and (almost) matching upper bounds for effectively computable approximations when the trigonometric moments of the measure are known.
 A second class of sum of squares polynomials is shown to interpolate the characteristic function on the support of the measure and to converge to zero outside.

 \medskip
	
 	\medskip
	
 \end{abstract}

\footnotetext[1]{
	Osnabr\"uck University, Institute of Mathematics\\
	\texttt{\{paul.catala,mathias.hockmann,stefan.kunis,markus.wageringel\}@uos.de}
	}
	
\section*{Introduction}

Data science in general and more specifically signal and image processing relies on mathematical methods, with the fast Fourier transform as the most prominent example.
Besides its favourable computational complexity, its success relies on the good approximation of smooth functions by trigonometric polynomials.
%
%
Mainly driven by specific applications, functions with additional properties together with their computational schemes have gained some attention: signals might for instance be sparse like in single molecule fluorescence microscopy \cite{moerner03}, or live on some other lower dimensional structure 
like microfilaments, again in bio-imaging. Such properties are well modeled by measures, which can express the underlying structure through the geometry of their support, \eg being discrete or singular continuous. This representation has in particular led to a better understanding of the sparse super-resolution problem \cite{candes13,denoyelle17,ehler19}, but has also proven useful in many more applications, such as phase retrieval in X-ray crystallography, or contour reconstruction in natural images. In this work, we consider measures supported on the torus. The available data then consists in trigonometric moments of low to moderate order, and one asks for the reconstruction or approximation of the measures.

\paragraph{Related work} For discrete measures, there is a large variety of methods that compute or approximate the parameters of the measure, e.g., parametric methods like Prony's method \cite{prony1795,PlTa14,kunis16,josz19,sauer17}, matrix pencil \cite{HuSa90,Moitra_15,ehler19}, ESPRIT \cite{roy89,andersson18,sahnoun17,li20} or MUSIC \cite{schmidt86,liao14}, or variational methods, such as TV-minimization via the Beurling LASSO \cite{deCastro12,candes13}, which can be challenging for higher spatial dimensions \cite{castro17,poon18} or larger polynomial degrees. The positive-dimensional case on the other hand is more involved. Specific curves in a two-dimensional domain are identified by the kernel of moment matrices in \cite{pan14,vetterli2016,ongie16}, more general discussions can be found in \cite{laurentrostalski2012} and \cite{wageringel2022:truncated}.
In another line of work, Christoffel functions offer interesting guarantees both in terms of support identification \cite{lasserre19} or approximation \emph{on the support} \cite{kroo12,marx21,pauwels20}, but, to the best of our knowledge, require strong regularity assumptions, and only come with separate guarantees on and outside the support of the measure.

\paragraph{Contributions} Following the seminal paper \cite{Mh19}, we introduce easily computable trigonometric polynomials to approximate an arbitrary measure on the $d$-dimensional torus.
In contrast to \cite{Mh19}, we provide tight bounds on the pointwise approximation error as well as with respect to the Wasserstein-1 distance, the latter scaling inverse linearly with respect to the polynomial degree (up to a logarithmic factor).
After setting up the notations, \cref{sec:approx} considers the approximation of measures by trigonometric polynomials. \cref{Thm_existence} proves the existence of a best approximation and provides a lower bound which is attained for the univariate case and is sharp within a factor $6d$ for spatial dimensions $d>1$.
The convolution of the measure with the Fejér kernel has a representation via the moment matrix of the measure and is shown to be a sum of squares for non-negative measures in \cref{thm:Tn}.
\cref{thm:W1p} proves a sharp upper bound for its approximation error being a $\log$-factor worse than the best approximation. \cref{Thm_lower_bound} and \cref{Rem_Jackson} discuss the saturation of this approximation and the removal of the $\log$-factor by using the Jackson kernel.
In the univariate case, the Wasserstein-1 distance of measures is realized as $L^1$-norm after convolution with the Bernoulli spline of degree~$1$ and this also allows for the uniqueness of the best approximation for absolutely continuous real measures.

Section \ref{sec:interp} studies another sum-of-squares trigonometric polynomial defined via the moment matrix of the measure, similarly suggested in \cite[Thm.~3.5]{kunis16} and \cite[Prop.~5.3]{ongie16} (and indeed closely related to the \emph{rational} function \cite[Eq.~(6)]{schmidt86}).
This polynomial interpolates the constant one function on the Zariski closure of the support of the measure and converges pointwisely to zero outside. 
\cref{thm:p1characterization} proves a variational characterisation as well as the interpolation property.
The pointwise convergence is proved in \cref{Thm_pointwise_conv} and \cref{thm:pointw_pos} for the discrete and singular continuous case, respectively.
The discrete case also allows for a weak convergence result in \cref{thm:p1weak}.
We end by illustrating the theoretical results by numerical examples in \cref{sec:num}.

\section{Preliminaries}
Let $d\in\N$, $1\le p\le \infty$ and let $|x-y|_p = \min_{k\in\Z^d} \norm{x-y+k}_p$ denote the wrap-around $p$-norm on $\T^d=[0,1)^d$. For $d=1$ these wrap-around distances coincide and we denote them by $|x-y|_1$ to distinguish from the absolute value.
Throughout this paper, let $\mu,\nu$ denote some complex Borel measures on $\T^d$ with finite total variation and normalization $\mu(\T^d)=\nu(\T^d)=1$.
We denote the set of all such measures by $\MM$ and restrict to the real signed and non-negative case by $\MM_{\R}$ and $\MM_+$, respectively.

A function has Lipschitz-constant at most $1$ if $|f(x)-f(y)|\le |x-y|_1$ for all $x,y\in\T^d$ and we denote this by the shorthand $\Lip(f)\le 1$.
Using the dual characterisation by Kantorovich-Rubinstein, the Wasserstein-1-distance of $\mu$ and $\nu$ is defined by
\begin{align*}
 W_1(\nu,\mu)
 =\sup_{\Lip(f)\leq 1} \left|\int_{\T^d} f(x) \;\diff(\nu-\mu)\left(x\right)\right|,
\end{align*}
for any $\mu,\nu\in\mathcal{M}$, and $\mu,\nu\in\mathcal{M}_+$ also admit the primal formulation 
\begin{align*}
    W_1(\nu,\mu)=\inf_{\pi} \int_{\T^{2d}} |x-y|_1 \diff \pi(x,y)
\end{align*}
where the infimum is taken over all couplings $\pi$ with marginals $\mu$ and $\nu$, respectively. We note in passing that the Wasserstein-1-distances for other $p$ norms on $\T^d$ are equivalent with lower and upper constant $1$ and $d^{1-1/p}$, respectively.
Moreover, the Wasserstein distance defines a metric induced by the norm
\begin{align*}
    \|\mu\|_{\Lip^*}=\sup_{f: \Lip(f)\leq 1, \|f\|_{\infty}\leq \frac{d}{2}} \left|\int_{\T^d} f(x) d\mu(x)\right| 
\end{align*}
which makes the space of Borel measures with finite total variation a Banach space.
By slight abuse of notation, we also write $W_1(p,\mu)$ in case the measure $\nu$ has density $p$, i.e., $\diff\nu(x)=p(x)\diff x$.

The Fourier coefficients or trigonometric moments of $\mu$ are given by
\begin{align*}
 \hat\mu(k)=\int_{\T^d} \eim{kx}\diff\mu(x),\qquad k\in\Z^d,
\end{align*}
and these are finite with $|\hat\mu(k)|\le \TV{\mu}$ and $\hat\mu(0)=1$.
We are interested in the reconstruction of the measure given these moments for indices $k\in\{-n,\hdots,n\}^d$.
Besides collecting them in a vector, we also set up the finite moment matrix
\begin{align}\label{eq:moment-matrix}
  T_n=\left(\hat\mu(k-\ell)\right)_{k,\ell\in [n]},\qquad [n]=\{0,\hdots,n\}^d,
\end{align}
and denote its singular value decomposition by $T_n=U_n \Sigma_n V_n^*$, $\Sigma_n=\diag\big(\sigma_j^{(n)}\big)_j$. From the data stored in $T_n$, we compute trigonometric approximations $q_n$ to the underlying measure and distinguish between pointwise convergence to the characteristic function of $\supp \mu$, i.e.
\begin{align*}
    \lim_{n\to\infty} q_n(x) = \begin{cases} 1, \quad & x\in\supp\mu, \\ 0, \quad & \text{else,} \end{cases}
\end{align*}
and weak convergence, i.e.
\begin{align*}
    \lim_{n\to\infty} \int_{\T^d} f(x) q_n(x) \diff x = \int_{\T^d} f(x) \diff\mu(x)
\end{align*}
for all continuous test functions $f\colon \T^d\to \C $. The latter is denoted by $q_n\rightharpoonup \mu$ as $n\to\infty$ and the space of test functions can be restricted to Lipschitz continuous test functions by Portmanteau's theorem. Moreover, $q_n\rightharpoonup \mu$ is equivalent to $\lim_{n\to\infty} W_1(q_n,\mu)=0$  on the bounded set $\T^d$ and we can quantify rates of weak convergence in terms of the Wasserstein distance (e.g.\,cf.\,\cite[Thm.\,6.9]{Villani_09}).
\section{Approximation}\label{sec:approx}
We give two introductory examples, prove a lower bound on a best approximation, and an upper bound on the easily computable approximation by convolution with the Fejér kernel.

\begin{example}\label{ex:1}
 Our first example for $d=1$ is the measure
 \begin{align}\label{eq:ex1}
  \mu=\frac13\delta_{\frac18}+\nu\in\MM_+,\quad \frac{\diff\nu}{\diff\lambda}(x)=\frac89\chi_{\left[\frac14,\frac58\right]}(x)+\frac{\sqrt{2}}{3}\left( \frac{1}{\sqrt{\abs{x-\frac78}}} - \sqrt{8}\right)\chi_{\left[\frac34,1\right]}(x),
 \end{align}
 where $\lambda$ denotes the Lebesgue measure, which obviously has singular and absolutely continuous parts including an integrable pole at $x=\frac78$.
 Given the first trigonometric moments, the Fourier partial sums
 \begin{align*}
     S_n\mu(x)=\sum_{k\in\Z^d, |k|\le n} \hat \mu(k) \eip{kx}
 \end{align*}
 might serve as a sequence of approximations.
 Another classical sequence of approximations is given by convolution with the Fejér kernels, see \cref{eq:pn} below.
 Both approximations for $n=19$ are shown in the left and right panel of \cref{fig:ExampleMeasure}, respectively.
\begin{figure}[htb]
    \centering
    \includegraphics[width=0.4\textwidth]{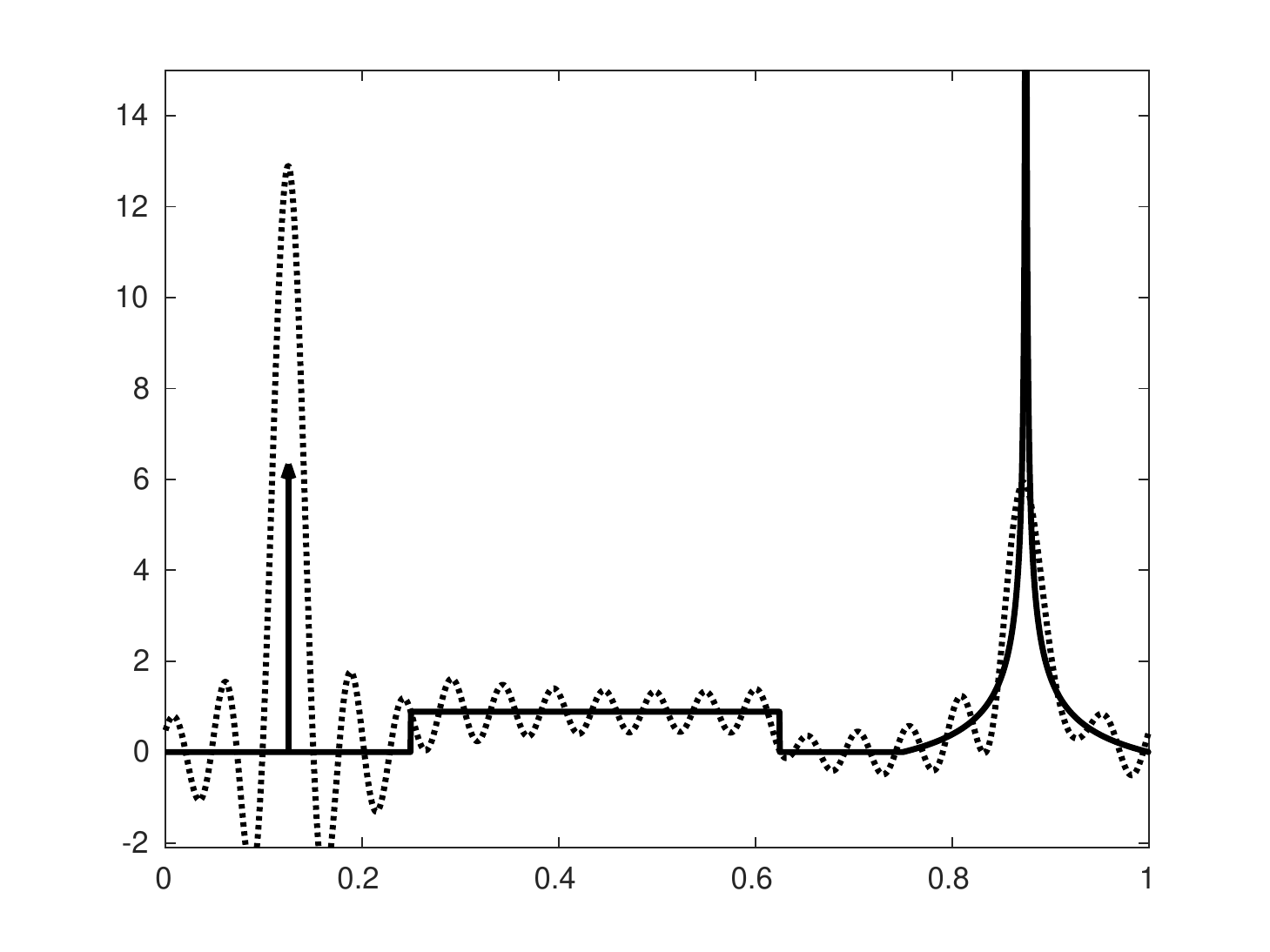}
    \includegraphics[width=0.4\textwidth]{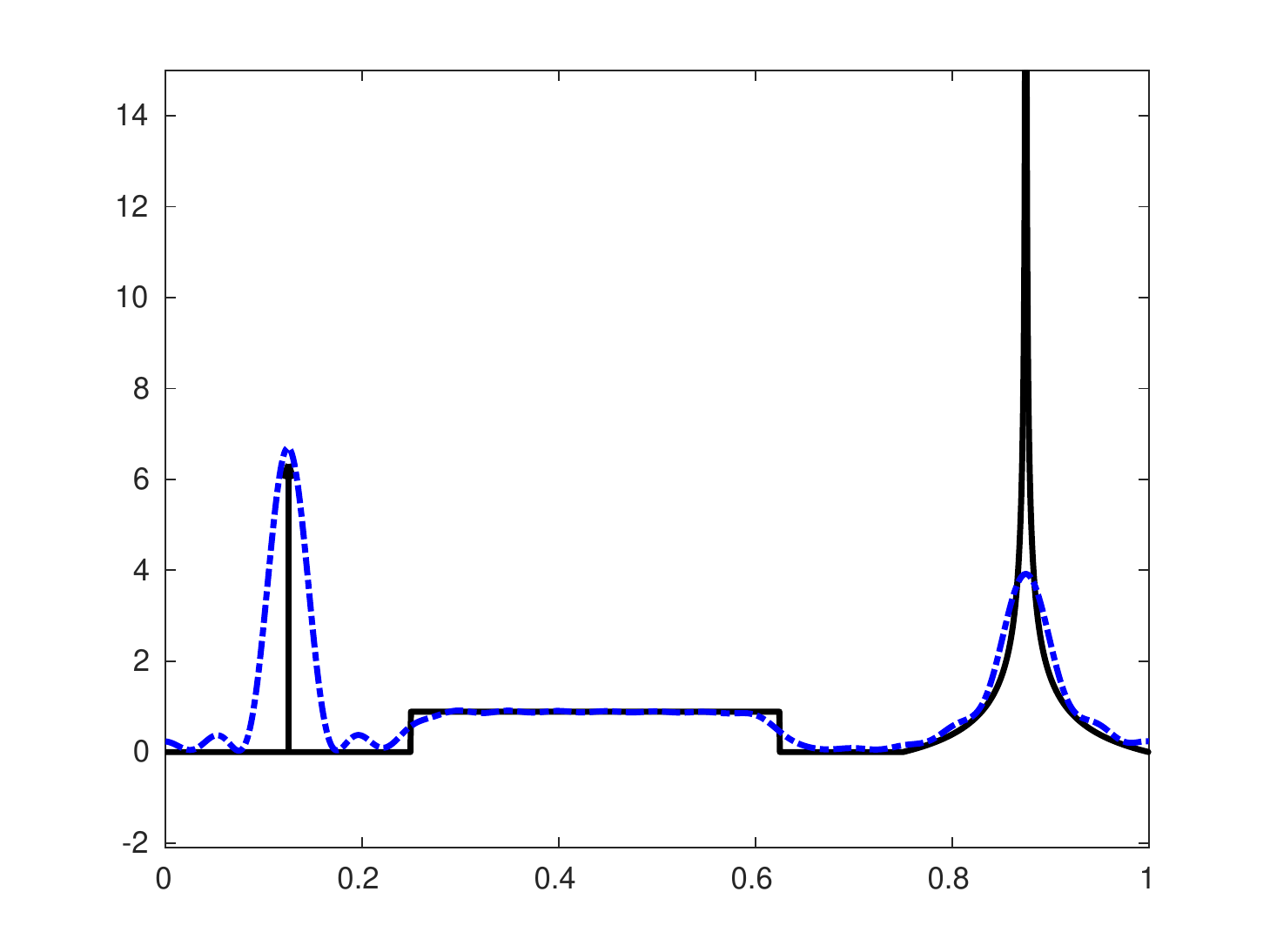}
    \caption{The example measure \cref{eq:ex1} and its approximations by the Fourier partial sum (left) and the convolution with the Fejér kernel (right). The weight $\frac13$ of the Dirac measure in $\frac18$ is displayed by an arrow of height $n/3$ for visibility.}
    \label{fig:ExampleMeasure}
\end{figure}

 Our second example is a singular continuous measure for $d=2$.
 We take $\mu=(2\pi r_0)^{-1}\delta_{C}\in\MM_+$ as the uniform measure on the circle 
\begin{align*}
C=\{x\in\T^2: |x|_2=r_0\}
\end{align*} 
for some radius $0<r_0<\frac{1}{2}$. The total variation of this measure is
\begin{align*}
    \TV{\mu} = \hat\mu(0)=\int_{\T^2} \diff \mu(x) = \frac{1}{2\pi r_0} \int_{C} \diff x = 1.
\end{align*}
Using the Poisson summation formula and a well-known representation of the Fourier transform of a radial function, we find
\begin{align}\label{eq:circlemoments}
\hat{\mu}(k)&=\int_{\T^2} \eim{kx} \diff\mu(x)= \frac{1}{r_0} \int_{0}^{\infty} r J_0(2\pi r \|k\|_2) \diff \delta_{r_0}(r)  = J_0(2\pi r_0 \|k\|_2)
\end{align}
for the trigonometric moments of $\mu$, where $J_0$ denotes the 0-th order Bessel function of the  first kind. These decay asymptotically with rate $\|k\|_2^{-1/2}$.
The Fourier partial sum as well as the convolution with the Fejér kernel for $n=29$ are shown with maximal contrast in the left and right panel of \cref{fig:ExampleMeasure2d}, respectively. 
\begin{figure}[htb]
    \centering
    \includegraphics[width=0.4\textwidth]{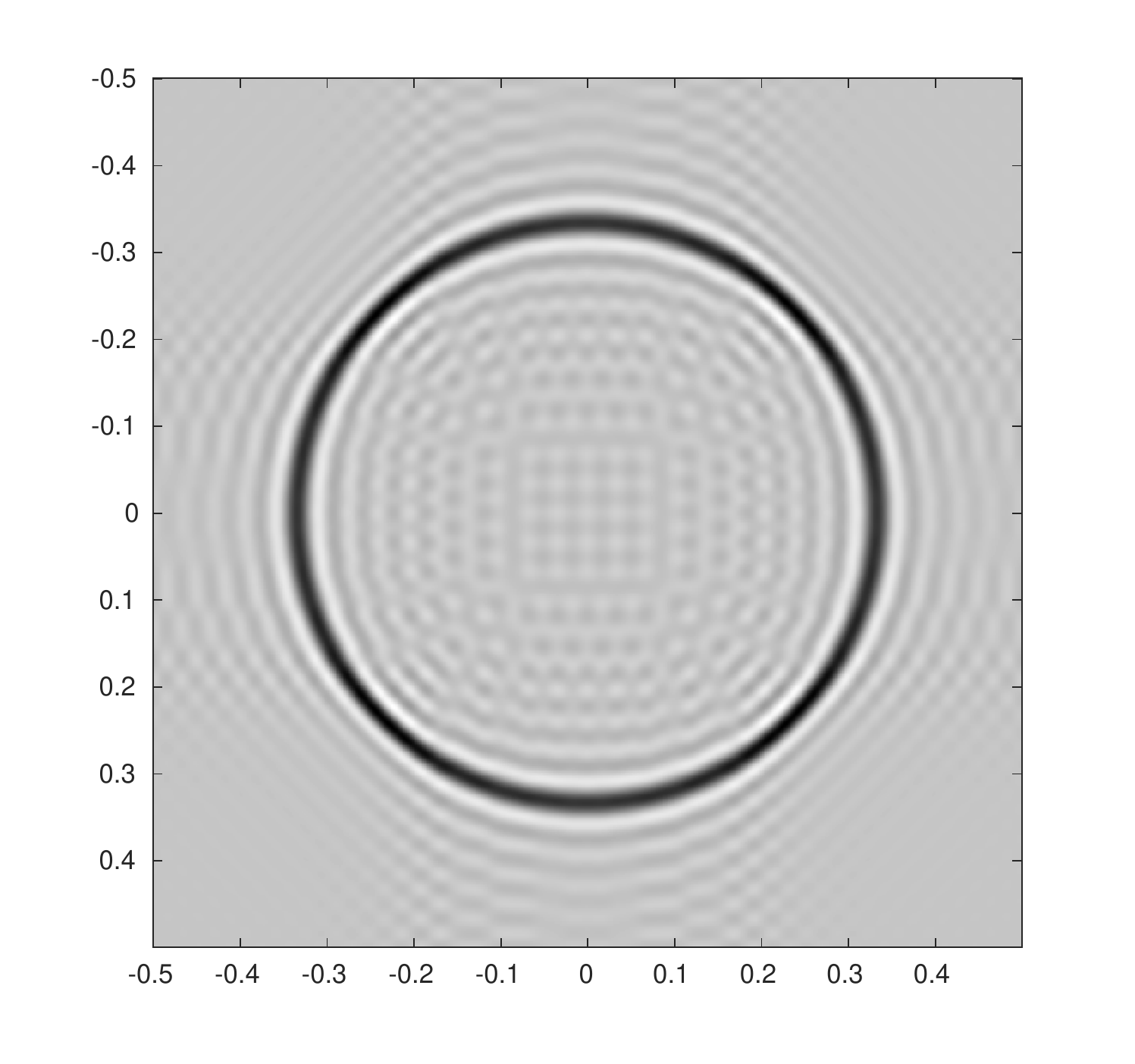}
    \includegraphics[width=0.4\textwidth]{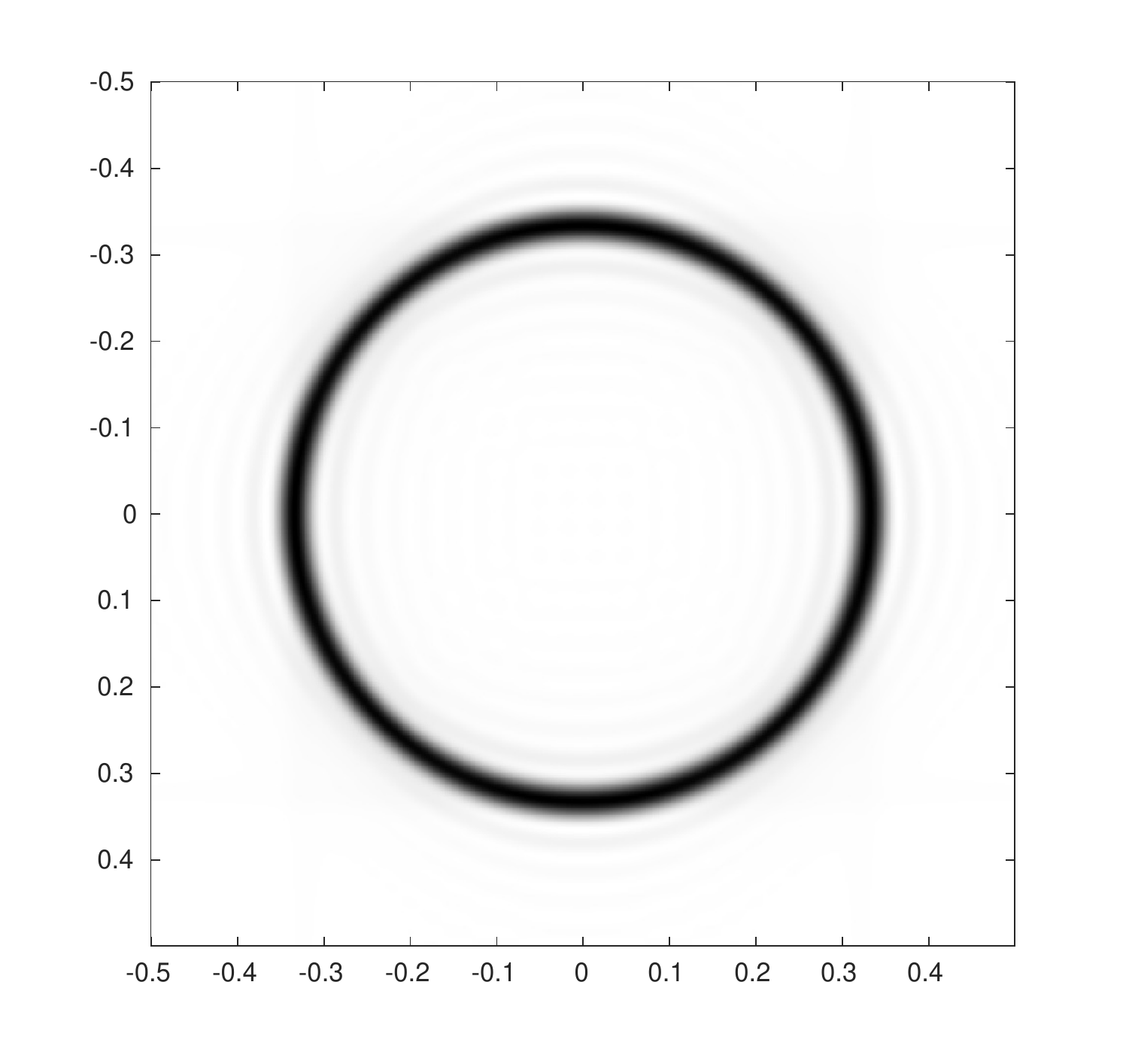}
    \caption{Uniform measure on a circle of radius $r_0=\frac13$ and its approximations by the Fourier partial sum (left) and the convolution with the Fejér kernel (right).}
    \label{fig:ExampleMeasure2d}
\end{figure}
\end{example}

\begin{thm} \label{Thm_existence}
For any $d,n\in\N$ and for every $\mu\in\MM$ there exists a polynomial of best approximation in the Wasserstein-1 distance. Moreover, we have
\begin{align*}
    \sup_{\mu\in\MM} \min_{\deg(p)\leq n} W_1(p,\mu) \geq \frac{1}{4(n+1)}.
\end{align*}
 
\end{thm}
\begin{proof}
We directly have existence of a best approximation by polynomials in the Banach space of Borel measures with finite total variation (e.g.\,cf.\,\cite[Thm.\,3.1.1]{DeVore93}). For the lower bound, we compute
    \begin{align*}
        \sup_{\mu\in\MM} \inf_{\deg(p)\leq n} W_1(p,\mu) &\geq \inf_{\deg(p)\leq n} W_1(p,\delta_0) \\
        &= \inf_{\deg(p)\leq n} \sup_{f: \Lip(f)\leq 1} \left|f(0)-\int_{\T^d} f(x)p(x) \diff x\right| \\
        &=\inf_{\deg(p)\leq n} \sup_{\genfrac{}{}{0pt}{}{f: \|f\|_{\infty}\leq \frac{d}{2}}{\Lip(f)\leq 1}} \left\|f-\check{p}*f\right\|_{\infty} \\
        &\geq \sup_{\genfrac{}{}{0pt}{}{f: \|f\|_{\infty}\leq \frac{d}{2}}{\Lip(f)\leq 1}} \inf_{\deg(p)\leq n} \|f-p\|_{\infty},
    \end{align*}
    where $\check{p}$ denotes the reflection of $p$.
    It remains to find the worst case error for the best approximation of a Lipschitz function by a trigonometric polynomial. This is well-understood for $d=1$ (cf.\,\cite{AK1937,Favard1937}) while we did not find a reference that and how $d>1$ is possible as well. Therefore, we show that the idea by \cite{Fisher77} for the case $d=1$ works also for $d>1$ in our situation. A main ingredient of Fishers proof is the duality relation
    \begin{align*}
        \inf_{x\in Y\subset X} \|x_0-x\| = \sup_{\genfrac{}{}{0pt}{}{\ell\in X^*}{\ell|_Y=0, \|\ell\|_{X^*}\leq 1} } |\ell(x_0)|
    \end{align*}
    for a Banach space $X$, $x_0\in X$, with subset $Y$ and dual space $X^*$. The second ingredient is given by the 1-periodic Bernoulli spline of degree 1
    \begin{align} \label{eq_Bernoulli_spline}
        \mathcal{B}_1(x)=\sum_{k\in\Z\setminus\{0\}} \frac{\eip{kx}}{2\pi\ii k} = \sum_{k=1}^\infty \frac{\sin(2\pi k x)}{\pi k} = \frac12-x
    \end{align}
    for $0< x \leq 1$. A Lipschitz continuous and 1-periodic function $f\colon \T\to \R$ with $\Lip(f)\leq 1$ has a derivative $f'$ almost everywhere and this derivative satisfies $\int_{\T} f'(s)=0$ by the periodicity of $f$. Therefore, it follows that
    \begin{align*}
        \left( f'*\mathcal{B}_1\right)(t)&=\int_{\T} f'(s) \mathcal{B}_1(t-s) \diff s \\
        &= -\int_0^{t} (t-s) f'(s) \diff s -\int_{t}^1 (t-s+1) f'(s) \diff s \\
        &=f(t)-\int_0^1 f(s) \diff s
    \end{align*}
    for $0<t,s\leq 1$. The dual space of the space of continuous periodic functions is the space of periodic finite regular Borel measures equipped with the total variation norm and the duality formulation gives
    \begin{align*}
        \sup_{\genfrac{}{}{0pt}{}{f:\T^d\to \R}{\|f\|_{\infty}, \Lip(f)\leq 1}}\inf_{\deg(p)\leq n} \|f-p\|_{\infty} 
        &= \sup_{\genfrac{}{}{0pt}{}{f:\T^d\to \R}{\|f\|_{\infty}, \Lip(f)\leq 1}}\sup_{\genfrac{}{}{0pt}{}{\hat{\mu}(k) =0, \|k\|_{\infty}\leq n}{\TV{\mu}\leq 1} } \left|\int_{\T^d} f(x) \diff \mu(x)\right|.
    \end{align*}
    Our main contribution to this result is the observation how to transfer the multivariate setting back to the univariate one. It is easy to verify that $f(x)=\frac{1}{d}\sum_{\ell=1}^d f_0(x_\ell)$ 
    for a univariate Lipschitz function $f_0$, $\Lip(f_0)\leq d$, $\|f_0\|_{\infty}\leq \frac{d}{2}$ fulfils the conditions for the Lipschitz function $f$. Additionally, $\mu^*=\frac{1}{d}\sum _{s=1}^d \mu_s$ with $\mu_s=\left(\bigotimes_{\ell\neq s} \lambda(x_\ell)\right) \otimes \mu_0^*(x_s)$,
    \begin{align*}
        \mu_0^*(x_s)=\frac{1}{2(n+1)} \sum_{j=0}^{2n+1} (-1)^j \delta_{j/(2n+2)}(x_s)
    \end{align*}
    and $\lambda$ being the Lebesgue measure on $\T$ is admissible. Since this choice of $\mu_s$ integrates $\int g \diff\mu_s=0$ if $g$ is constant with respect to $x_s$ (and the same holds for constant univariate functions integrated against $\mu_0^*$), we obtain
    \begin{align*}
        \sup_{\genfrac{}{}{0pt}{}{f:\T^d\to \R}{\|f\|_{\infty}, \Lip(f)\leq 1}}\inf_{\deg(p)\leq n} \|f-p\|_{\infty} &\geq \sup_{\genfrac{}{}{0pt}{}{f_0:\T\to \R}{\|f_0\|_{\infty}\leq 1, \Lip(f_0)\leq d}} \left|\frac{1}{d^2} \sum_{s,\ell=1}^d \int_{\T^d} f_0(x_\ell) \diff \mu_s(x)\right| \\
        &=\sup_{\genfrac{}{}{0pt}{}{f_0:\T\to \R}{\|f_0\|_{\infty}\leq 1, \Lip(f_0)\leq d}} \left|\frac{1}{d^2} \sum_{\ell=1}^d \int_{\T} f_0(x_\ell) \diff \mu_0^*(x_\ell)\right| \\
        &=\sup_{\genfrac{}{}{0pt}{}{f_0:\T\to \R}{\|f_0\|_{\infty}\leq 1, \Lip(f_0)\leq d}} \frac{1}{d} \left| \int_{\T}  f_0'(s) \left(\int_{\T}\mathcal{B}_1(t-s) \diff \mu_0^*(t)\right) \diff s\right|.
    \end{align*}
   We denote $\mathcal{B}_{\mu*}(s)= \int_{\T}\mathcal{B}_1(t-s) \diff \mu_0^*(t)$ and observe $\int_{\T} \mathcal{B}_{\mu*}(s) \diff s=0$. Moreover, $\mu_0^*$ has moments $\hat{\mu}_0^*(k) = 1$ for $k\in (n+1)\left(2\Z+1\right)$ and $\hat{\mu}_0^*(k)=0$ otherwise. Together with the Fourier representation \eqref{eq_Bernoulli_spline} of $\mathcal{B}_1$ this gives
   \begin{align*}
       \mathcal{B}_{\mu*}(s) &= \sum_{k\in\Z\setminus\{0\}} \frac{\eip{k(n+1)x}}{2\pi\ii k(n+1)} - \sum_{k\in\Z\setminus\{0\}} \frac{\eip{2k(n+1)x}}{2\pi\ii 2k(n+1)} \\
       &=  \frac{1}{n+1} \mathcal{B}_1((n+1)s)- \frac{1}{2n+2} \mathcal{B}_1((2n+2)s).
   \end{align*}
   Hence, taking $f'_0(s)=\pm d$ depending on the sign of $\mathcal{B}_{\mu*}(s)$ is possible as $\|f_0\|_{\infty}\leq \frac{d}{2(n+1)}< \frac{d}{2}$ with this choice. Finally, we end up with
   \begin{align*}
        \sup_{\genfrac{}{}{0pt}{}{f:\T^d\to \R}{\|f\|_{\infty}, \Lip(f)\leq 1}}\inf_{\deg(p)\leq n} \|f-p\|_{\infty} & \geq   \int_{\T} \left|   \mathcal{B}_{\mu*}(s)\right|\diff s \\
        &= \frac{1}{n+1} \int_{\T} \left|\mathcal{B}_1((n+1)s)-\frac12\mathcal{B}_1((2n+2)s)\right| \diff s \\
        &= \frac{1}{n+1} \int_{\T} \left|\mathcal{B}_1(s)-\frac{1}{2}\mathcal{B}_1(2s)\right| \diff s \\
        &= \frac{1}{4(n+1)}
   \end{align*}
   and this was the claim.
\end{proof}

Let the Fejér kernel $F_n\colon\T\to\R$ and by slight abuse of notation also its multivariate version be given by
\begin{align*}
    F_n(x)=\sum_{k=-n}^n \left(1-\frac{|k|}{n+1}\right)\eip{kx}
    =\frac{1}{n+1}\left(\frac{\sin(n+1)\pi x}{\sin\pi x}\right)^2
\end{align*}
and $F_n(x_1,\hdots,x_d)=F_n(x_1)\cdot\hdots\cdot F_n(x_d)$, respectively.
The main object of study now is the approximation
\begin{align}\label{eq:pn}
    p_n(x) &=\left(F_n*\mu\right)(x) =\int_{\T^d} F_n(x-y) \diff\mu(y),
\end{align}
two examples are given in \cref{ex:1}.
We start by noting that the $p_n$ can be expressed in terms of the moment matrix and preserves non-negativity and normalization.
\begin{lemma}\label{thm:Tn}
 Let $d,n\in\N$, $\mu\in\MM$, and the moment matrix $T_n$ in \cref{eq:moment-matrix} be given, then the approximation \cref{eq:pn} fulfils
 \begin{align}
  p_n(x)
  &=\frac{e_n(x)^* T_n e_n(x)}{(n+1)^d},\qquad &e_n(x)&=\left(\eim{kx}\right)_{k\in[n]},\nonumber \\
  &=\frac{1}{(n+1)^d} \sum_{j=1}^N \sigma^{(n)}_j u^{(n)}_j(x)\overline{v^{(n)}_j(x)},\qquad &u^{(n)}_j(x)&=e_n(x)^* U_n,\;v^{(n)}_j(x)=V_n^* e_n(x) \nonumber
 \intertext{If $\mu\in\MM_{\R}$, then the moment matrix is Hermitian.
 If $\mu\in\MM_+$, then the moment matrix is positive semi-definite, $\|p_n\|_{L^1}=\TV{\mu}=1$, and}
  p_n(x)&=\frac{1}{(n+1)^d} \sum_{j=1}^N \sigma^{(n)}_j \left|u^{(n)}_j(x)\right|^2. \label{eq:pnuj}
 \end{align}
\end{lemma}
\begin{proof}
 Let $q\in\C^{(n+1)^d}$, then direct computation shows
 \begin{align*}
  q^* T_n q = \sum_{k,l\in [n]} \overline{q_k} \left(\int_{\T^d} \eim{(k-\ell)y}\diff\mu(y)\right) q_\ell
  =\int_{\T^d} \left|\sum_{k\in [n]} q_k \eip{ky}\right|^2 \diff\mu(y) \ge 0.
 \end{align*}
 Choosing $q=e_n(x)$ yields the second claim and by interchanging the order of integration and noting that the value of the inner integral is independent of $y$ also
 \begin{align*}
    \|p_n\|_{L^1}
    =\frac{1}{(n+1)^{d}} \int_{\T^d} 
    \int_{\T^d} \left|\sum_{k\in [n]} \eip{k(y-x)}\right|^2  \diff x \diff\mu(y)
    =\mu(\T^d)=\TV{\mu}.
\end{align*}
\end{proof}

Our next goal is a quantitative approximation result, for which we need the following preparatory lemma.
This result can be found in qualitative form e.g.~in \cite[Lemma 1.6.4]{BuNe71}.

\begin{lemma}\label{le:Fn}
Let $n,d\in\N$, then we have
\begin{align*}
    \frac{d}{\pi^2} \left(\frac{\log(n+2)}{n+1}+\frac{1}{n+3}\right)
    \leq
    \int_{\T^d} F_n(x) |x|_1 \diff x
    \leq \frac{d}{\pi^2}\frac{\log(n+1) + 3}{n}.
\end{align*}
\end{lemma}
\begin{proof}
 First note that
 \begin{align*}
  \int_{\T^d} \prod_{s=1}^d F_n(x_s)  \sum_{\ell=1}^d |x_\ell|_1 \diff x
    = \sum_{\ell=1}^d \int_{\T^d} \prod_{s=1}^d F_n(x_s) |x_\ell|_1 \diff x
    = d \int_{\T} F_n(x) |x|_1 \diff x
 \end{align*}
 such that it is sufficient to consider the univariate case.
 With the representation $F_n(x)=1+2\sum_{k=1}^n \left(1-\frac{k}{n+1}\right) \cos(2\pi k x)$ we find after elementary integration
\begin{align*}
   \int_{\T} F_n(x) |x|_1 \diff x &= 2 \int_0^{1/2} x+2\sum_{k=1}^n \left(1-\frac{k}{n+1}\right) \cos(2\pi k x)x \d x \\
    &=2 \left[\frac{1}{8}-\frac{1}{(n+1)\pi^2}\sum_{j=0}^{\lfloor \frac{n}{2} \rfloor} \frac{n-2j}{(2j+1)^2} \right] \\
    &=2\left[\frac{1}{\pi^2} \sum_{j=\lfloor \frac{n}{2} \rfloor+1}^{\infty} \frac{1}{(2j+1)^2} + \frac{1}{(n+1)\pi^2} \sum_{j=0}^{\lfloor \frac{n}{2} \rfloor} \frac{1}{2j+1} \right] \\
    &\leq \frac{2}{\pi^2} \int_{\lfloor \frac{n}{2} \rfloor}^\infty \frac{1}{(2y+1)^2} \diff y + \frac{2}{(n+1)\pi^2} \left(1+ \int_{0}^{\lfloor \frac{n}{2} \rfloor} \frac{1}{2y+1} \diff y\right) \\
    &\leq\frac{1}{\pi^2n}+ \frac{1}{(n+1)\pi^2} \left(2+ \log(n+1)\right) \\
    &\leq \frac{3+\log(n+1)}{\pi^2 n},
\end{align*}
since we have that $\sum_{j=0}^{\infty} \frac{1}{(2j+1)^2}=\frac{\pi^2}{8}$. The lower bound follows similarly by bounding the series from the previous calculation by integrals from below.
\end{proof}

\begin{thm}\label{thm:W1p}
Let $d,n\in\N$ and $\mu\in\MM$, then the measure with density $p_n$ converges weakly to $\mu$ with
\begin{align*}
    W_1(p_n,\mu)\le \frac{d}{\pi^2}\frac{\log(n+1) + 3}{n} \cdot \TV{\mu} ,
\intertext{which is sharp since $\mu\in\MM_+$ implies $\TV{\mu}=1$ and}
\sup_{\mu\in\MM} W_1(p_n,\mu) \geq \frac{d}{\pi^2} \left(\frac{\log(n+2)}{n+1}+\frac{1}{n+3}\right).
\end{align*}
\end{thm}
\begin{proof}
We compute
\begin{align*}
    W_1(p_n,\mu)
    &=\sup_{\Lip(f)\leq 1} \left|\langle F_n*\mu,f\rangle-\langle \mu,f\rangle\right|\\
    &=\sup_{\Lip(f)\leq 1} \left|\langle \mu,F_n*f-f\rangle\right|\\
    &\leq \sup_{\Lip(f)\leq 1} \int_{\T^d} \int_{\T^d} F_n(x)  \left|f(y-x)-f(y)\right| \diff x \diff |\mu|(y) \\
    &\leq \TV{\mu} \int_{\T^d} F_n(x)  |x|_1 \diff x,
\end{align*}
note that both inequalities become equalities when choosing $\mu=\delta_0$ and $f(x)=|x|_1$, and then apply \cref{le:Fn}.
We note in passing that $W_1(F_n,\delta_0)=
    \int_{\T^d} F_n(x) |x|_1 \diff x$.
\end{proof}

\begin{example}
 \cref{thm:W1p} gives a worst case lower bound and, on the other hand, the Lebesgue measure is approximated by $F_n*\lambda=\lambda$ without any error.
 We may thus ask how well a measure $\diff\mu=w(x) \diff x$ with smooth (non-negative) density might be approximated. If we choose the analytical density $w(x)=1+\cos(2\pi x)$, then $F_n*w(x)-w(x)=\cos(2\pi x)/(n+1)$ and, by testing with the Lipschitz function $f(x)=\cos(2\pi x)/(2\pi)$, we see that
 \begin{align*}
  W_1(F_n*w,w)
  &\ge \frac{1}{2\pi(n+1)}\int_{\T}\cos^2(2\pi x)\diff x=\frac{1}{4\pi(n+1)}.
 \end{align*}
 This effect is called saturation (e.g.\,cf.\,\cite{BuNe71}).
\end{example}

In greater generality, such a lower bound holds for each measure individually and can be inferred by a nice relationship between the Wasserstein distance and a discrepancy, cf.~\cite{EhGrNeSt21}.
\begin{thm} \label{Thm_lower_bound}
For each individual measure $\mu\in\MM$ different from the Lebesgue measure, there is a constant $c>0$ such that 
\begin{align*}
   W_1(p_n,\mu)\geq  \frac{c}{n+1}
\end{align*}
holds for all $n\in\N$.
\end{thm}
\begin{proof}
Let $\hat{h}\in\ell^2(\Z^d)$, $\hat{h}(k)\in\R\setminus\{0\}$, $\hat{h}(k)=\hat{h}(-k)$,
and consider the reproducing kernel Hilbert space
\begin{align*}
    H=\{f\in L^2(\T^d): \sum_{k\in\Z^d} |\hat{h}(k)|^{-2} |\hat{f}(k)|^2 <\infty\},\qquad
    \|f\|_{H}^2=\sum_{k\in\Z^d} |\hat{h}(k)|^{-2} |\hat{f}(k)|^2.
\end{align*}
Given two measures $\mu,\nu$, their discrepancy (which also depends on the space $H$) is defined by
\begin{align*}
    \mathcal{D}(\mu,\nu)=\sup_{\|f\|_{H}\leq 1} \left|\int_{\T^d} f~\diff (\mu- \nu)\right|
\end{align*}
and fulfils by the geometric-arithmetic inequality
\begin{align*}
    \mathcal{D}(p_n,\mu)^2
    &=\sum_{\|k\|_{\infty}\leq n} |\hat{h}(k)|^{2} \left|1-\prod_{\ell=1}^d \left(1-\frac{|k_\ell|}{n+1}\right)\right|^2 |\hat{\mu}(k)|^2 + \sum_{\|k\|_{\infty}> n} |\hat{h}(k)|^{2}|\hat{\mu}(k)|^2 \\
    &\geq \sum_{\|k\|_{\infty}\leq n} |\hat{h}(k)|^{2} \left|\frac{\|k\|_1}{d(n+1)}\right|^2 |\hat{\mu}(k)|^2 + \sum_{\|k\|_{\infty}> n} |\hat{h}(k)|^{2}|\hat{\mu}(k)|^2 \\
    &=\sum_{\|k\|_{\infty}\leq n} |\hat{h}(k)|^{2} \left|\frac{\|k\|_1}{d(n+1)}\right|^2 |\hat{\mu}(k)-\hat{\lambda}(k)|^2 + \sum_{\|k\|_{\infty}> n} |\hat{h}(k)|^{2}|\hat{\mu}_k-\hat{\lambda}(k)|^2 \\
    &\geq \frac{1}{d^2(n+1)^2} \|h*(\mu-\lambda)\|_{L^2(\T^d)}^2
\end{align*}
where $h(x)=\sum_{k\in\Z^d} \hat{h}(k) \eip{kx}$ and $\lambda$ denotes the Lebesgue measure with $\hat \lambda(0)=1$ and $\hat \lambda(k)=0$ for $k\in\Z^d\setminus\{0\}$.

Our second ingredient is a Lipschitz estimate: If $f\in H$ with $\|f\|_{H}\leq 1$, then 
\begin{align*}
    |f(y)-f(y+x)|^2
    &=\left|\sum_{k\in\Z^d} \hat{f}(k) \left(\eip{ky}-\eip{k(y+x)}\right)\right|^2\\
    &\leq \|f\|_{H}^2 \sum_{k\in\Z^d} \left|\eip{ky}-\eip{k(y+x)}\right|^2 |\hat{h}(k)|^2\\
    &\leq 2\left(K(0)-K(x)\right),
\end{align*}
where $K(x)=\sum_{k\in\Z^d} |\hat{h}(k)|^2 \eip{kx}=(h*h)(x)$ denotes the so-called reproducing kernel of the space $H$.
If this kernel is $K(x_1,\hdots,x_d)=h^{[4]}(x_1)\cdot\hdots\cdot h^{[4]}(x_d)$ for some univariate function $h^{[4]}\in C^2(\T)$, $\left(h^{[4]}\right)'(0)=0$, we find by a telescoping sum and direct calculation
\begin{align*}
    K(0)-K(x)&= \prod_{\ell=1}^dh^{[4]}(0)-\prod_{\ell=1}^d h^{[4]}(x_\ell) \\
    &\leq \sum_{\ell=1}^d \left(h^{[4]}(0)^\ell\prod_{k=1}^{d-\ell} h^{[4]}(x_k)-h^{[4]}(0)^{\ell-1}\prod_{k=1}^{d-\ell+1} h^{[4]}(x_k)\right)\\
    &\leq \sum_{\ell=1}^d \|h^{[4]}\|^{d-1}_{\infty} \left[h^{[4]}(0)-h^{[4]}(x_\ell)\right] \\
    &\leq \frac12 \|h^{[4]}\|^{d-1}_{\infty} \left\|\left(h^{[4]}\right)''\right\|_{\infty} |x|^2.
\end{align*}

To make a specific choice, let $a\in(0,\frac{1}{8})$ be some irrational number and set $h^{[2]}= \chi_{[-a,a]}*\chi_{[-a,a]}$ as the convolution of the indicator function on $[-a,a]$ with itself,
$h^{[4]}=h^{[2]}*h^{[2]}$, and $h(x_1,\hdots,x_d)=h^{[2]}(x_1)\cdot\hdots\cdot h^{[2]}(x_d)$.
Since the space of Lipschitz test functions is at least as large as the reproducing kernel Hilbert space, we derive
\begin{align*}
    W_1(p_n,\mu) \geq \frac{1}{2}\cdot \left(\frac{\sqrt{3}}{4}\right)^{d-1} a^{1-\frac{3}{2}d} \mathcal{D}(p_n,\mu)
    \geq \frac{1\cdot \left(\frac{\sqrt{3}}{4}\right)^{d-1} a^{1-\frac{3}{2}d}}{d(n+1)} \|h*(\mu-\lambda)\|_{L^2(\T^d)}.
\end{align*}
Since $a$ is irrational, we can directly see by Parseval's theorem that $\|h*(\mu-\lambda)\|_{L^2(\T^d)}=0$ if and only if $\mu=\lambda$. For $\mu\neq\lambda$, we obtain the statement with a positive constant $c$ depending on the measure $\mu$, the constant $a$, and the spatial dimension $d$.
\end{proof}

\begin{remark} \label{Rem_Jackson}
 The gap between upper and lower bounds can be narrowed by choosing another convolution kernel, which then however does not allow for the representation in \cref{thm:Tn}.
 The Jackson kernel
 \begin{align*}
     J_{2m-2}(x)=\frac{3}{m(2m^2+1)} \frac{\sin^4(m\pi x)}{\sin^4(\pi x)}, \quad m\in\N,
 \end{align*}
 has degree $n=2m-2$ and satisfies
 \begin{align*}
     \int_{\T} J_n(x) |x|_1 \diff x \leq \frac{6}{m(2m^2+1)} \left[\int_0^{1/2m} m^4 x \diff x+ \int_{1/2m}^{\infty} \frac{1}{16x^3} \diff x\right] \leq \frac{3m}{2(2m^2+1)} \le \frac{3}{2(n+2)}.
 \end{align*}
 Analogously to \cref{thm:W1p}, we get
 \begin{align}
    \label{eq:W1-jackson}
     W_1(J_n*\mu,\mu)\leq \frac{3}{2} \frac{d \cdot \TV{\mu}}{n+2},
 \end{align}
 which still is an approximate factor 6 worse than the lower bound in the univariate case. A factor 3 is due to the above estimate and a factor 2 seems to indicate that the Jackson kernel is not optimal.
 Moreover, upper and lower bound differ by a factor $d$ in the multivariate case which might be due to the used norms or our proof techniques.
 
 We mention at this point that 
 \begin{align*}
     W_1(K_n*\mu,\mu)\leq \TV{\mu} W_1(\delta_0,K_n),&&\text{and}&&
     W_1(\delta_0,K_n)\le \int_{\T^d} |K_n(x)| |x|_1 \diff x
 \end{align*}
 for any kernel $K_n$ and measure $\mu\in\MM$ with equality in the second inequality if $K_n$ is nonnegative.
\end{remark}

\subsection{Univariate case}

In one variable, the question of uniqueness of the best approximation can be equivalently characterised by the uniqueness of the best approximation in $L^1(\T)$ and thus allows for the following lemma.

\begin{lemma}[Best approximation in the univariate case] For $d=1$, any absolutely continuous real measure admits a unique best approximation by a polynomial of degree $n\in\N$ with respect to the Wasserstein-1 distance.\label{Lem_W11D}
\end{lemma}
\begin{proof}
Let $\mu,\nu\in\mathcal{M}_\R$ and $\mathcal{B}_1$ denote the Bernoulli spline of degree 1 from the proof of \cref{Thm_existence}, then we have
\begin{align*}
    W_1(\nu,\mu)&=\sup_{f: \Lip(f)\leq 1} \left|\int_{\T} f(x)\left[ \diff \nu(x) -\diff\mu(x)\right]\right| \\
    &=\sup_{f: \Lip(f)\leq 1} \left|\int_{\T}\int_{\T} f'(t)\mathcal{B}_1(x-t)\left[\diff \nu(x)  -\diff\mu(x)\right]\diff t\right| \\
    &=\sup_{f: \Lip(f)\leq 1} \left|\int_{\T} f'(t)\left[(\mathcal{B}_1*\nu)(t)-(\mathcal{B}_1*\mu)(t)\right] \diff t\right|.
\end{align*}
Since the integral over $f'$ is zero by the periodicity of $f$, any $c\in\R$ yields
\begin{align*}
    \left|\int_{\T} f'(t)\left[(\mathcal{B}_1*\nu)(t)-(\mathcal{B}_1*\mu)(t)\right] \diff t\right| &= \left|\int_{\T} f'(t)\left[(\mathcal{B}_1*\nu)(t)-(\mathcal{B}_1*\mu)(t)-c\right] \diff t\right| \\ &\leq \inf_{c\in\R} \int_{\T} \left|(\mathcal{B}_1*\nu)(t)-(\mathcal{B}_1*\mu)(t)-c\right| \diff t.
\end{align*} 
On the other hand, choosing $c^*$ such that $\{t: (\mathcal{B}_1*\nu)(t)-(\mathcal{B}_1*\mu)(t)>c^*\}$ and $\{t: (\mathcal{B}_1*\nu)(t)-(\mathcal{B}_1*\mu)(t)<c^*\}$ have the same mass yields
\begin{align*}
     \sup_{f: \Lip(f)\leq 1} \left|\int_{\T} f'(t)\left[(\mathcal{B}_1*\nu)(t)-(\mathcal{B}_1*\mu)(t)-c^*\right] \diff t\right| \geq \int_{\T} \left|(\mathcal{B}_1*\nu)(t)-(\mathcal{B}_1*\mu)(t)-c^*\right| \diff t
\end{align*}
by taking $f$ with $f'(t)=\pm 1$ depending on the sign of the term in brackets. Because of $\int_\T (\mathcal{B}_1*\nu-\mathcal{B}_1*\mu)(t)\diff t =0$, this gives $c^*=0$ and thus
\begin{align}\label{eq:W1L1}
     W_1(\nu,\mu)=\int_{\T} \left|(\mathcal{B}_1*\nu)(t)-(\mathcal{B}_1*\mu)(t)\right| \diff t.
\end{align}

We proceed by computing explicitly
\begin{align}\label{eq:B1mu}
    (\mathcal{B}_1*\mu)(t)
    = \frac{\mu([0,t))+\mu([0,t])}{2} - \mu([0,1))\left(t +\frac12\right)+\int_{[0,1)} x \diff\mu(x).
\end{align}
If $\mu$ does not give mass to single points, we have that $\mathcal{B}_1*\mu$
is continuous and hence there exists a unique best $L^1$-approximation $\tilde{p}$ (e.g.\,cf.\,\cite[Thm.\,3.10.9]{DeVore93}) which defines $p^*$ uniquely by $\tilde{p}=\mathcal{B}_1*p^*$. 
\end{proof}

\begin{example} \label{Ex:Unique_best}
Uniqueness and non-uniqueness of $L^1$ approximation is discussed in some detail in \cite{Moskona95,Dryanov12} and we note the following:
\begin{enumerate}[(i)]
    \item For $\mu=\frac{1}{2}\delta_0-\frac{1}{2}\delta_{1/2}+\lambda\in\MM_{\R}$ where $\lambda$ is again the Lebesgue measure, one finds
    \begin{align*}
        (\mathcal{B}_1*\mu)(t)=\frac12\left(\mathcal{B}_1(t)-\mathcal{B}_1\Big(t-\frac12\Big)\right)= \begin{cases} 0,&\quad t=0,\\ \frac14, & \quad t\in\big(0,\frac12\big), \\ 0,& \quad t=\frac12, \\ -\frac14,&\quad t\in\big(\frac12,1\big). \end{cases}
    \end{align*}
    As proved in \cite[Thm.~5.1]{Moskona95}, this function does not have a unique $L^1$ approximation and thus $\mu$ does not admit a unique approximation by a polynomial due to \cref{Lem_W11D} for even $n$. 
    \item For $\mu=\delta_0$ one has $\mathcal{B}_1*\mu=\mathcal{B}_1$ and according to \cite[Lem.~2.2]{Moskona95} this function with only one jump has a unique best $L^1$-approximation given by the interpolation polynomial
    \begin{align*}
        \tilde{p}(x)
        =\sum_{j=1}^n \frac{1}{2n+2} \cot\left(\frac{j\pi }{2n+2}\right) \sin(2\pi jx).
    \end{align*}
    Deconvolving $\tilde{p}=\mathcal{B}_1*p^*$ gives
    \begin{align*}
        p^*(x)
        = 1 + \sum_{j=1}^n \frac{j\pi }{n+1} \cot\left(\frac{j\pi}{2n+2}\right) \cos(2\pi jx)
    \end{align*}
    as the unique best approximation to $\delta_0$. Since the error of the best $L^1$ approximation of $\mathcal{B}_1$ is known from a theorem by Favard \cite{Favard1937} (e.g.\,this is mentioned in \cite[p.\,213]{DeVore93}), we can compute
    \begin{align*}
        W_1(\delta_0,p^*)=\left\|\mathcal{B}_1-\mathcal{B}_1*p^*\right\|_{L^1(\T)}=\frac{1}{4(n+1)}
    \end{align*}
    and this reveals that the estimates in the proof of \cref{Thm_existence} are sharp.
\end{enumerate}
\end{example}

\cref{fig:BestApprox} and \cref{tab:W1p} summarize our findings on the approximation of $\delta_0$. The best approximation $p^*$ as well as the Dirichlet kernel $D_n(x)=\sin(2n+1)\pi x /\sin\pi x$ are signed with small full width at half maximum but positive and negative oscillations at the sides. The latter might be seen as an unwanted artifact in applications. The approximations given by the Fejér and the Jackson kernel are non-negative.
\begin{figure}[htb]
        \centering
        \includegraphics[width=\columnwidth]{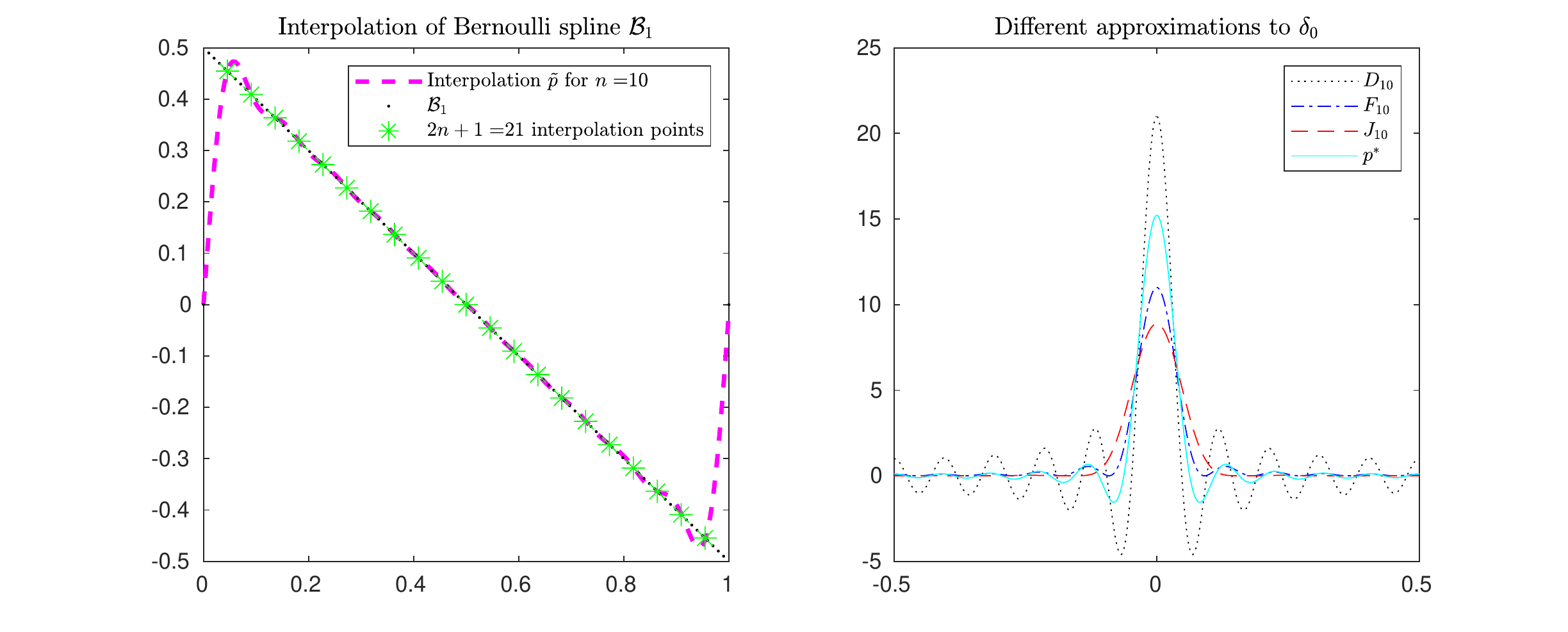}
        \caption{Interpolation of $\mathcal{B}_1$ (left) and comparison of different polynomial approximations of degree $n=10$ to $\delta_0$ (right).}
        \label{fig:BestApprox}
    \end{figure}
For completeness, we note that the Dirichlet kernel is the Fourier partial sum of $\delta_0$ and allows for the estimate
    \begin{align*}
       W_1(\delta_0,D_n)\leq W_1(\delta_0,p^*) + W_1(p^*,D_n) \leq \left(1 + \|D_n\|_1\right) W_1(\delta_0,p^*) \leq \frac{\frac{4}{\pi^2} \log(n)+5}{4(n+1)}
   \end{align*} 
   which relies on $W_1(p^*,D_n)= W_1(D_n*p^*,D_n*\delta_0)\leq \|D_n\|_1 W_1(\delta_0,p^*)$,
   the well known bound on the Lebesgue constant \cite[Prop.\,1.2.3]{BuNe71}, and \cref{Ex:Unique_best} (iii).
    \begin{table}[h]
        \centering
        \begin{tabular}{|c|c|l|}
        \hline
         Trigonometric polynomial  & Sign of polynomial & $W_1(\delta_0,K_n)$  \\\hline
        \rule{0pt}{15pt}   Dirichlet $D_n$  & signed & $\leq  \frac{\frac{1}{\pi^2}\log(n)+\frac{5}{4}}{n+1}$ (\cref{Ex:Unique_best} (iii)) \\[8pt]\hline
  \rule{0pt}{15pt} Fejér $F_n$ & nonnegative &  $\leq \frac{1}{\pi^2}\frac{\log(n+1) + 3}{n}$ (\cref{thm:W1p}) \\[8pt]\hline
   \rule{0pt}{15pt} Jackson $J_n$, $n$ even & nonnegative & $\leq \frac{3}{2} \frac{1}{n+2}$ (\cref{Rem_Jackson})\\[8pt]\hline
   \rule{0pt}{15pt} Best approximation $p^*$ & signed & $= \frac{1}{4(n+1)}$ (\cref{Ex:Unique_best} (ii)) \\[8pt]\hline
        \end{tabular}
        \caption{Convergence rates of different trigonometric polynomials approximating $\delta_0$. }
        \label{tab:W1p}
    \end{table}
    
\begin{remark}
We close by some remarks which are specific for the univariate setting:
\begin{enumerate}[(i)]
    \item We stress that \eqref{eq:W1L1} in the proof of the \cref{Lem_W11D} allows to compute the Wasserstein distance as an $L^1$-distance for real signed univariate measures.
    Similarly, this allows to compute the so-called star discrepancy $\|\nu([0,\cdot))\|_{\infty}$ as suggested in \cite[eq. (2.1) and (2.2)]{Mh19}. However note that \eqref{eq:B1mu} has some additional term such that $\nu=\frac{1}{2}\delta_0-\frac{1}{2}\delta_{1/2}$ with $\nu(\T)=0$ gives
    \begin{align*}
        \|\nu([0,\cdot))\|_{\infty} = \frac12  \neq  \frac14 = \|\mathcal{B}_1*\nu\|_{\infty}
    \end{align*}
    and thus \cite[eq. (2.1) and (2.2)]{Mh19} needs some adjustment.

   On the real line, we indeed have $\mu((-\infty,x])=(H*\mu)(x)$ for the Heaviside function
\begin{align*}
    H(x)=\begin{cases} 1, & \quad x\geq 0, \\ 0, & \quad \text{else,} \end{cases}
\end{align*}
   such that the Wasserstein distance again can be computed via
\begin{align*}
    W_1(\mu,\nu)=\int_{\R} \left|\mu((-\infty,x])-\nu((-\infty,x])\right| \diff x
    = \|H*(\mu-\nu)\|_{L^1(\R)},
\end{align*}
see e.g.~\cite[Prop.\,2.17]{Santambrogio_15}.
 \cref{Lem_W11D} might be considered as the periodic analogue of this result.
    
\item One can relate our work to a main result in \cite{Mh19}. As \cref{Lem_W11D} reformulates the Wasserstein distance of two univariate measures in terms of the $L^1$-distance of their convolution with the Bernoulli spline, one can view this Bernoulli spline as a kernel of type $\beta=1$ following the notation of \cite{Mh19}. Thus, one can take $p=1,p'=\infty$ in \cite[Thm.\,4.1]{Mh19} yielding that the Wasserstein distance between a measure $\mu$ and its trigonometric approximation is bounded from above by $c/n$. The latter agrees with our \cref{Rem_Jackson} which additionally gives an explicit and small constant. 

\item The observation that the construction of $p^*$ for $\delta_0$ is possible via FFT's might lead to the idea to construct near-best approximations to any measure $\mu$ by interpolating $\mathcal{B}_1*\mu$ by some $\tilde{p}$ and to obtain the polynomial $p$ of near best approximation which satisfies $\tilde{p}=\mathcal{B}_1*p$ by multiplying with the Fourier coefficients of the Bernoulli spline $\mathcal{B}_1$. A first problem would be that the limited knowledge of moments only allows to interpolate the partial Fourier sum $S_n(\mathcal{B}_1*\mu)$ which does not converge to $\mathcal{B}_1*\mu$ uniformly as $n\to\infty$ for discrete $\mu$. Secondly, the near-best approximation $p$ cannot be expected to be nonnegative for a nonnegative measure $\mu$ which is another drawback compared to convolution with nonnegative kernels like the Fejér or Jackson kernel.

\item Finally note that kernels $K_n$ with stronger localization and `smoother' Fourier coefficients, e.g.~higher powers of the Fejér kernel, allow to improve the rate beyond $n^{-1}$ if the measure has a smooth density $w$. This can be seen from partial integration
 \begin{align*}
  W_1(K_n * w, w)
  &=\sup_{\Lip(f)\le 1} \left|\int_{\T} \left(K_n*f(y)-f(y)\right) w(y) \diff y\right|\\
  &=\sup_{\psi, \Lip(\psi')\le 1} \left|\int_{\T} \left(K_n*\psi(y)-\psi(y)\right) w'(y) \diff y\right|
 \end{align*}
 and the above arguments. However note that from a practical perspective, this asks for a-priori smoothness assumptions on the measure to choose a suitable kernel.
\end{enumerate}
\end{remark}

\section{Interpolation}\label{sec:interp}

Using the singular functions of the moment matrix as in \cref{eq:pnuj} and with $r=r(n)=\rank T_n$, we define noise- and signal-functions $p_{0,n},p_{1,n}:\T^d\to [0,1]$,
\begin{align}\label{eq:p1}
 p_{1,n}(x)=\frac{1}{N} \sum_{j=1}^r \left|u^{(n)}_j(x)\right|^2,\qquad
 p_{0,n}(x)=\frac{1}{N} \sum_{j=r+1}^N \left|u^{(n)}_j(x)\right|^2,
\end{align}
respectively.
In what follows, we suppose that $V\subseteq\T^d$ denotes the smallest set
containing $\supp\mu$ that is the zero-locus of some unknown trigonometric polynomial,
i.\,e.\ $V \supseteq \supp \mu$ is the Zariski closure of the support. 
We show pointwise convergence
\begin{align}\label{eq:chiV}
 p_{1,n}(x) \xrightarrow{n\to\infty} \chi_V(x)=\begin{cases} 1, & x\in V,\\ 0, & \text{otherwise},\end{cases}
\end{align}
to the characteristic function of this set. Clearly this cannot be uniform and our first result shows interpolation of the value $1$ for finite $n$ as well as a variational characterization.

\begin{thm}\label{thm:p1characterization}
 Let $d,n\in\N$, $\mu\in\MM$, and suppose $V(\ker T_n)=V\subseteq\T^d$, where $V(\ker T_n)$ is the set consisting of the common roots of all the polynomials in $\ker T_n$. Then $p_{0,n}(x)+p_{1,n}(x)=1$ for all $x\in\T^d$. In particular, we have
 \begin{equation}\label{eq:p1interpolating1}
  \begin{aligned}
  p_{1,n}(x)\begin{cases}
  =1, & \text{\ if $ x\in V$},\\
  <1, & \text{otherwise}.
  \end{cases}
  \end{aligned}
 \end{equation}
 If $V\subsetneq\T^d$, the variational characterization
 \begin{align}\label{eq:var:p0}
    p_{0,n}(x)=\max_{p} \frac{|p(x)|^2}{N \|p\|_{L^2}^2} 
\end{align}
 holds, where the maximum is subject to all trigonometric polynomials
 $p\in\langle \eip{kx}: k\in [n]\rangle$, $p\ne 0$, such that $p(y)=0$ for all $y\in V$.
\end{thm}
\begin{proof}
  We have
  \begin{equation}\label{eq:p1p0sos}
    p_{1,n}(x) + p_{0,n}(x)
    = \frac{1}{N} \sum_{j=1}^N \abs{u_j^{(n)}(x)}^2
    = \frac{1}{N} e_n(x)^* U_n U_n^* e_n(x)
    = \frac{1}{N} e_n(x)^* e_n(x)
    = 1,
  \end{equation}
  so in particular $p_{1,n}(x) \in [0,1]$.
  Since $V(\ker T_n) = V$ and $\ker T_n = \langle u_{j+1}^{(n)},\dots,u_N^{(n)}\rangle$,
  it follows that the polynomials $u_{j+1}^{(n)}(x),\dots,u_N^{(n)}(x)$ vanish on $V$,
  so $p_{1,n}(x) = 1$ for all $x\in V$.
  Conversely, if $x\in\T^d$ such that $p_{1,n}(x) = 1$,
  we claim that $x\in V$.
  Indeed, we have $1 - p_{1,n}(x) = \sum_{j=r+1}^N \abs{u_j^{(n)}(x)}^2 = 0$,
  so it follows that $x$ lies in the vanishing set of $u_{j+1}^{(n)}(x),\dots,u_N^{(n)}(x)$,
  so $x\in V(\ker T_n) = V$.

  For the variational characterization,
  first note that the admissible polynomials belong to $\ker T_n \setminus\{0\}$,
  which is non-empty since $V(\ker T_n) = V \subsetneq \T^d$.
  Let $q \coloneqq U_0 U_0\adj e_n(x) \in \ker T_n$,
  where $U_0$ denotes a matrix whose columns form an orthonormal basis of $\ker T_n$.
  For all $p \in \ker T_n$, we have
  \[
    q\adj p = e_n(x)\adj U_0 U_0\adj p = e_n(x)\adj p = p(x).
  \]
  In particular, note that
  \begin{equation}\label{eq:p1variational:qq}
    \norm[2]{q}^2 =
    q\adj q = q(x) = e_n(x)\adj U_0 U_0\adj e_n(x) = N p_{0,n}(x).
  \end{equation}
  Therefore, by the Cauchy--Schwarz inequality, it follows that
  \[
    \abs{p(x)}^2 = \abs{q\adj p}^2
    \le \norm[2]{q}^2 \cdot \norm[2]{p}^2
    = N p_{0,n}(x) \norm[2]{p}^2.
  \]
  Hence, we have
  \[
    p_{0,n}(x)
    \ge \max_{p\in\ker T_n\setminus\{0\}} \frac{\abs{p(x)}^2}{N \norm[2]{p}^2}
    \ge \frac{\abs{q(x)}^2}{N \norm[2]{q}^2}
    = p_{0,n}(x),
  \]
  if $q\ne 0$.
  The first inequality also holds when $q=0$,
  in which case the result follows due to \eqref{eq:p1variational:qq}.
\end{proof}

Note that \eqref{eq:p1interpolating1} and \eqref{eq:p1p0sos} generalise and strengthen
\cite[Propositions~5.2,\ 5.3]{ongie16}
from algebraic curves on the two-dimensional torus
to algebraic varieties of arbitrary dimension.

\begin{remark}
  The hypothesis $V(\ker T_n) = V$ in \cref{thm:p1characterization}
  is satisfied for all sufficiently large $n$
  if $\mu\in\MM$ is finitely-supported, see e.g.~\cite{KuNaSt22}.
  Similarly, it holds for sufficiently large $n$
  if $\mu\in\MM_+$;
  see for example \cite[Theorem~2.10]{laurentrostalski2012}
  or \cite[Proposition~4.10]{wageringel2022:truncated}.

  If $\mu\in\MM$ is not finitely-supported,
  then $V(\ker T_n) = V$ can fail to hold for any $n\in\N$
  (cf.~\cite[Example~4.9]{wageringel2022:truncated}).
  In this case, it is possible to rephrase the hypothesis
  in terms of a non-square moment matrix of suitable size
  (cf.~\cite[Theorem~4.3]{wageringel2022:truncated})
  to obtain a statement similar to \cref{thm:p1characterization}.
\end{remark}

\begin{example}
 For $\mu=\delta_0$, we have $p_{1,n}(x)=F_n(x)/(n+1)^d$ and the proofs of Theorem \ref{Thm_pointwise_conv} and \ref{thm:p1weak} also show that $p_{1,n}$ is close to a sum of normalized Fejér kernels for arbitrary discrete measures.
 
 Moreover note that the above construction only yields the Zariski closure for general support sets, i.\,e., $p_{1,n}(x)=1$ for all $x\in\T^d$ and $n\in\N$ holds
 for every support set that is not contained in the zero locus of any non-trivial trigonometric polynomial.

 A singular continuous measure with support on the zero locus of a specific trigonometric polynomial is discussed as a numerical example in \cref{sec:num}.
\end{example}

\subsection{Zero-dimensional situation}

If the measure is given by
\begin{align*}
    \mu=\sum_{j=1}^r \lambda_j \delta_{x_j}
\end{align*}
with support $V=\supp\mu=\{x_1,\hdots,x_r\}\subset\T^d$ and complex weights $\Lambda=\diag(\lambda_1,\hdots,\lambda_r)$, then the support is zero-dimensional and the moment matrix allows for the Vandermonde factorisation
\begin{align*}
 T_n=A_n^*\Lambda A_n,\qquad A_n=(\eip{kx_j})_{j=1,\hdots,r; k\in[n]}\in\C^{r\times N},
\end{align*}
which will be instrumental.

\begin{thm}[Pointwise convergence] \label{Thm_pointwise_conv}
  Let $\mu = \sum_{j=1}^r \lambda_j \delta_{x_j}$ be a discrete measure
  and let $x\in\T^d$ such that $x\ne x_j$ for all $1\le j\le r$. If $n+1>4d/\min_{j\ne\ell}|x_j-x_\ell|_{\infty}$, then
  \begin{align*}
      p_{1,n}(x) \le \frac{1}{(n+1)^2} \cdot \frac{|\lambda_{\max}|}{|\lambda_{\min}|} \cdot \frac{1}{3}\sum_{j=1}^r \frac{1}{ |x - x_j|_{\infty}^2},
  \end{align*}
  in particular, this implies the pointwise convergence \eqref{eq:chiV}.
  Moreover, if the support of $\mu$ fulfils $n+1>2\sqrt{d}/\min_{j\ne\ell}|x_j-x_\ell|_{\infty}$ and $\min_j |x-x_j|_{\infty}\le \sqrt{d}/(n+1)$, then
  \begin{align*}
      p_{1,n}(x) \le 1- \frac{3^{d-1}(2d-1)}{2^d d^{2+d/2}} \cdot (n+1)^2 \cdot \min_j |x-x_j|_{\infty}^2.
  \end{align*}
\end{thm}
\begin{proof}
  Comparing \eqref{eq:p1} with \cref{thm:Tn} and using \eqref{eq:pn} yields
  \begin{align*}
    p_{1,n}(x)
    &\le \frac{1}{N} \sum_{j=1}^r \frac{\sigma_j}{\sigma_{\min}} \abs{u^{(n)}_j(x)}^2
    = \frac{1}{\smin} \sum_{j=1}^r |\lambda_j| F_n(x-x_j).
  \end{align*}
  The final estimate follows from
  \begin{align*}
    F_n(x-x_j) \le \frac{(n+1)^{d-1}}{(n+1) \sin^2\lr{\pi |x - x_j|_{\infty}}} \le \frac{(n+1)^{d-2}}{4 |x - x_j|_{\infty}^2}
  \end{align*}
  and $\sigma_{\min}\ge 0.8 (n+1)^d |\lambda_{\min}|$, see \cite[Thm.~4.2]{KuNaSt22}.
  
  Regarding the second estimate, consider the Vandermonde matrix
  \begin{align*}
    \tilde{A}_{n,x}=\begin{bmatrix} e_n(x_1) & \cdots & e_n(x_r) & e_n(x)\end{bmatrix} \in\C^{N\times (r+1)}
  \end{align*}
  and note that its pseudo-inverse gives rise to the Lagrange polynomial
  \begin{align*}
    \ell_{r+1}(y)=e_{r+1}^*\tilde{A}_{n,x}^\dagger e_n(y).
  \end{align*}
  satisfying $\ell_{r+1}(x_j)=0$ for $j=1,\hdots,r$ and $\ell_{r+1}(x)=1$.
  We compute
 \begin{align*}
    \|\ell_{r+1}\|_{L^2}^2 &=\int_{\T^d} |e_{r+1}^*\tilde{A}_{n,x}^\dagger e_n(y)\rangle|^2 \d y \\
    &=\int_{\T^d} |\langle \tilde{A}_{n,x}^{\dagger *} e_{r+1},  e_n(y)\rangle|^2 \d y
    =\|\tilde{A}_{n,x}^{\dagger *} e_{r+1}\|_2^2\le \sigma_{\min}(\tilde{A}_{n,x})^{-2}
 \end{align*}
 and use \cref{thm:p1characterization} to bound 
 \begin{align*}
    1-p_{1,n}(x)=p_{0,n}(x)=\max_{p} \frac{|p(x)|^2}{N \|p\|_{L^2}^2} \ge
    \frac{|\ell_{r+1}(x)|^2}{N \|\ell_{r+1}\|_{L^2}^2}
    \geq \frac{\sigma_{\min}(\tilde{A}_{n,x})^{2}}{N}.
 \end{align*}
 The assertion follows from known estimates on the smallest singular value for the Vandermonde matrix with pairwise clustering nodes, see \cite[Corollary\,3.20]{HoKu21}.
\end{proof}

\begin{remark} \label{remark_p1_upper_bound}
  Actually, \cref{Thm_pointwise_conv} shows the correct orders in $n$ and $\min_j |x-x_j|_{\infty}^2$ in the upper bound of $p_{1,n}(x)$. 
  First note that $1-p_{1,n}$ and all its partial derivatives of order 1 vanish on $x_1,\hdots,x_r$.
  For fixed $x\in\T^d$, the Taylor expansion in $x_0=\argmin_j |x-x_j|_{\infty}$ thus gives
  \begin{align*}
      1-p_{1,n}(x)&= \frac{1}{2} (x-x_0)^\top H_x(\xi) (x-x_0),\qquad H_x(\xi)=\left(\frac{-\partial^2 p_{1,n}}{\partial x_r\partial x_s}\left(\xi\right)\right)_{1\le r,s\le d}\\
      &\leq \|H_x(\xi)\|_{\mathrm{F}} \cdot \frac{d}{2} \cdot \min_j |x-x_j|_{\infty}^2 \\
      &\leq 2\pi^2d^2n^2\min_j |x-x_j|_{\infty}^2,
  \end{align*}
  where the last inequality uses an entrywise Bernstein inequality and $\|p_{1,n}\|_{\infty}=1$.
\end{remark}

\begin{figure}[htb]
    \includegraphics[width=0.47\columnwidth]{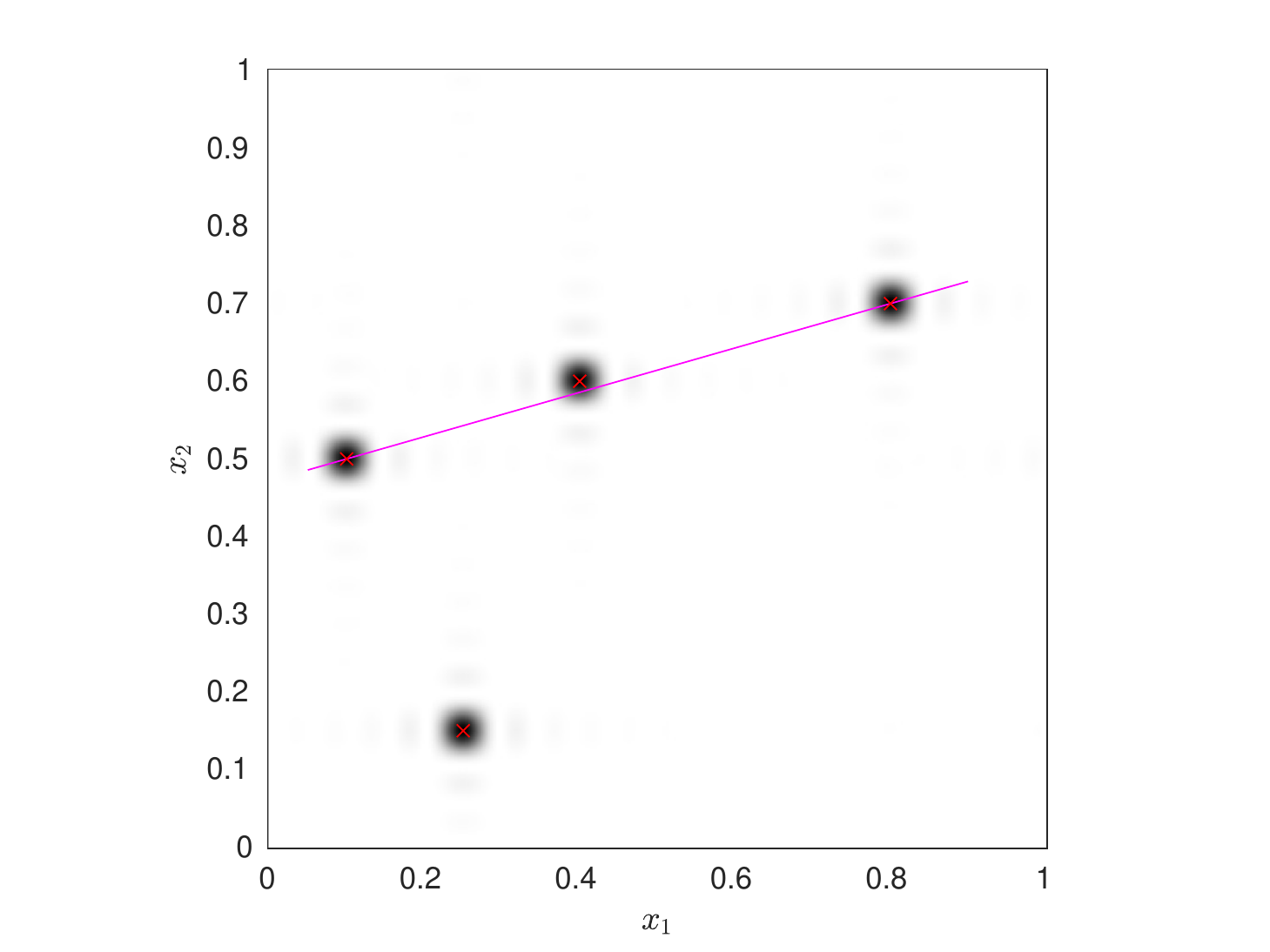}
    \includegraphics[width=0.52\columnwidth]{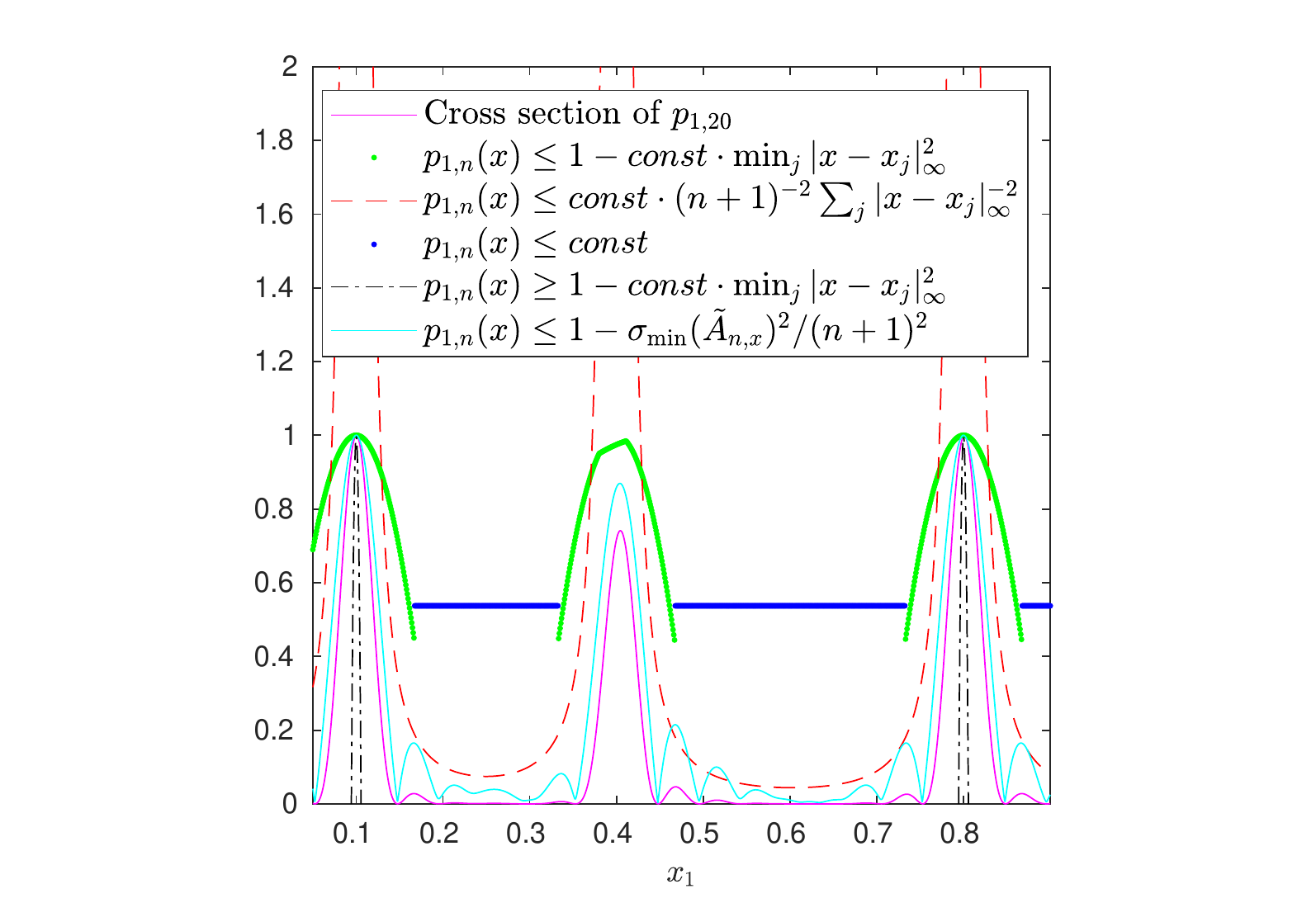}
    \caption{Summary of the bounds on $p_{1,n}$ from \cref{Thm_pointwise_conv} and \cref{remark_p1_upper_bound} for $d=2$, $n=20$, and a discrete measure $\mu$ supported on four points. The polynomial $p_{1,20}$ was evaluated on a grid in $\T^2$ and interpolated on the magenta cross section (left), while the bounds on $p_{1,20}$ on this cross section are displayed (right). We see that specifically the bound $1-\sigma_{\min}(\tilde{A}_{n,x})^2/N$ from the proof of \cref{Thm_pointwise_conv} reproduces the behaviour of $p_{1,n}$. The constant upper bound on $p_{1,n}$ away from the support of $\mu$ can be derived by using estimates for $\sigma_{\min}(\tilde{A}_{n,x})$ in the case of separated nodes.}
    \label{fig:Bounds_on_p1}
\end{figure}

\begin{lemma}[Convergence of singular values]\label{svalsconvergence}
Let $\mu=\sum_{j=1}^r \lambda_j \delta_{x_j}$ be a discrete complex measure whose weights are ordered non-increasingly with respect to their absolute value.
Assume that $(n+1)\min_{j\ne\ell}|x_j-x_{\ell}|_{\infty} > d$, then the singular values $\sigma_j$ of the moment matrix $T_n$ fulfil
\begin{align*}
    \left||\lambda_j|-\frac{\sigma_j}{N}\right|
    \leq \frac{1}{n+1}\cdot \frac{\abs{\lambda_{1}} \lr{1+\sqrt{\e}} r}{2\min_{j\ne\ell}|x_j-x_{\ell}|_{\infty}},\qquad
    j=1,\hdots,r.
\end{align*}
\end{lemma}
\begin{proof}
  With the polar decomposition $\frac{1}{\sqrt{N}} A_n\adj = P H$,
  where $P\in\C^{N\times r}$ is unitary and $H\in\C^{r\times r}$ is positive-definite,
  we have that $\abs{\lambda_1}\ge\cdots\ge\abs{\lambda_r}$ are the singular values
  of the matrix $P \Lambda P^*$.
  Therefore, for the singular values of $T_n = A_n\adj \Lambda A_n$, we obtain
  \begin{align*}
    \max_{1\le j\le r} \abs{\frac{\sigma_j}{N} - \abs{\lambda_j}}
    &\le \norm[2]{\frac{1}{N} T_n - P \Lambda P^*}
    = \norm[2]{H \Lambda H^* - \Lambda}\\
    &\le \norm[2]{H \Lambda \lr{H - \idmat{r}}} + \norm[2]{\lr{H - \idmat{r}}\Lambda}\\
    &\le \abs{\lambda_1} \lr{\norm[2]{H} + 1} \norm[2]{H - \idmat{r}}\\
    &\le \abs{\lambda_1} \lr{\norm[2]{H} + 1} \norm[2]{(H+\idmat{r})^{-1}} \norm[2]{H^2 - \idmat{r}}\\
    &\le \abs{\lambda_1} \frac{\frac{1}{\sqrt{N}} \smax(A_n) + 1}{\frac{1}{\sqrt{N}} \smin(A_n) + 1} \normf{\frac{1}{N} A_n A_n^* - \idmat{r}},
  \end{align*}
  where the first inequality is due to \cite[Theorem~2.2.8]{bjorck2015}.
  Each entry of the matrix $\frac{1}{N} A_n A_n^* - \idmat{r}$ is a modified Dirichlet kernel and can be bounded uniformly by
  \begin{align*}
    \normf{\frac{1}{N}A_n^*A_n - \idmat{r}} &= \frac{1}{N} \left(\sum_{j=1}^r\sum_{l\neq j} \left|\sum_{k\in [n]}\eip{k(x_l-x_j)}\right|^2\right)^{1/2}
    \leq  \frac{r}{N} \cdot \frac{(n+1)^{d-1}}{2 \min_{j\ne\ell}|x_j-x_{\ell}|_{\infty}}.
  \end{align*}
  Moreover, since $(n+1)\min_{j\ne\ell}|x_j-x_{\ell}|_{\infty}> d$, it follows from \cite[Lemma~2.1]{KuNaSt22} that
  \[
    \frac{1}{\sqrt{N}} \smax(A_n)
    \le \sqrt{\lr{1 + \frac{1}{d}}^d}
    \le \sqrt{\e}.
    \qedhere
  \]
\end{proof}

\begin{thm}\label{thm:p1weak}
  Normalizing differently, we have
  \begin{equation*}
    \frac{p_{1,n}}{\|p_{1,n}\|_{L^1}} \rightharpoonup \tilde\mu=\frac{1}{r}\sum_{j=1}^r \delta_{x_j}
  \end{equation*}
  as $n\to\infty$.
\end{thm}
\begin{proof}
  First note that $\|p_{1,n}\|_{L^1}=r/N$.
  We define $\tilde p_n=F_n*\tilde\mu$ and observe that
  for any continuous function $f$ on $\T^d$
  we have
  \begin{align*}
    &\mathrel{\hphantom{\le}}\abs{\int_{\T^d} \frac{p_{1,n}(x)}{\|p_{1,n}\|_{L^1}} f(x) \d x - \frac{1}{r}\sum_{j=1}^r f(x_j)}\\
    &\le
    \abs{\int_{\T^d} \lr{\frac{p_{1,n}(x)}{\|p_{1,n}\|_{L^1}} - \tilde p_n(x)} f(x) \d x}
    + \abs{\int_{\T^d} \tilde p_n(x)f(x)\d x - \frac{1}{r}\sum_{j=1}^r f(x_j)} \\
    &\le
    \normlonetm{\frac{N}{r}p_{1,n} - \tilde p_n} \normlinftm{f}
    + \abs{\int_{\T^d} f\d(F_n*\tilde\mu) - \int_{\T^d} f\d\tilde\mu},
  \end{align*}
  so, by \cref{thm:W1p},
  it is enough to show that $\normlonetm{\frac{N}{r} p_{1,n} - \tilde p_n}$ converges to zero for $n\to\infty$.

  If $n$ is sufficiently large, then by \eqref{eq:pnuj} we can write
  $\tilde p_n(x)
  = \frac{1}{N} e_n(x)\adj \tilde U \tilde\Sigma \tilde U\adj e_n(x)$
  where $\tilde\Sigma\in\C^{r\times r}$ denotes the diagonal matrix
  consisting of non-zero singular values
  and $\tilde U \in \C^{N\times r}$ denotes the corresponding singular vector matrix
  of the moment matrix of $\tilde\mu$.

  As $p_{1,n}$ only depends on the signal space of the moment matrix $T_n$ of $\mu$,
  which agrees with the signal space of the moment matrix of $\tilde\mu$,
  it follows by \eqref{eq:p1} that
  $p_{1,n}(x)
  = \frac{1}{N} e_n(x)\adj \tilde U \tilde U\adj e_n(x)$
  and thus
  \[
    \abs{\frac{N}{r} p_{1,n}(x) - \tilde p_n(x)}
    = \abs{e_n(x)\adj \tilde U \lr{\frac{\idmat{r}}{r} - \frac{\tilde\Sigma}{N}} \tilde U\adj e_n(x)}
    \le \norm[2]{e_n(x) \tilde U}^2 \norm[2]{\frac{1}{r}\idmat{r} - \frac{1}{N}\tilde\Sigma}.
  \]
  Since $\int_{\T^d} \norm[2]{e_n(x) \tilde U}^2 \d x = N \|p_{1,n}\|_{L^1} = r$ is constant,
  the result follows from \cref{svalsconvergence}.
\end{proof}

\subsection{Positive-dimensional situation}
For a measure $\mu$ whose support is an algebraic variety, we derive a pointwise convergence rate $p_{1,n}(x)=\mathcal{O}\left(n^{-1}\right)$ outside of the variety in \cref{thm:pointw_pos} and this proves \eqref{eq:chiV}. It is not clear whether this is already optimal, as we found $\mathcal{O}\left(n^{-2}\right)$ as an approximation rate in the case of a discrete measure. 
\begin{thm} \label{thm:pointw_pos}
  Let $y\in \T^d$ and
  let $g \in\langle \eip{\scalarprod{k}{x}} \mid k\in [m]\rangle$ be a trigonometric polynomial of max-degree $m$
  such that $g(y)\ne 0$ and $g$ vanishes on $\supp\mu$.
  Then
  \[
    p_{1,n+m}(y)
    \le 1 - \frac{(n+1)^d}{(n+m +1)^d} \cdot \frac{\abs{g\lr{y}}^2}{\lr{F_n * \abs{g}^2}(y)}
    \le \frac{\|g\|_{L^2}^2}{|g(y)|^2} \frac{m(4m+2)^d}{n+1} + \frac{d m}{\lr{n+m + 1}},
  \]
  for $n\in\N$, $n\geq m$.
\end{thm}
\begin{proof}
  Set $N_n = (n+1)^d$ for $n\in\N$
  and define the trigonometric polynomial
  $p(x) = e_{n,y}(x) g(x)$ of max-degree $n+m$,
  where $e_{n,y}(x) \coloneqq e_n(x)\adj e_n(y)$.
  Furthermore, we define $f(x)\coloneqq \abs{g(x)}^2$.
  Then
  \[
    \abs{p(x)}^2 = N_n F_n(x-y) f(x),
  \]
  for all $x\in\T^d$.
  On the other hand,
  \[
    \normltm{p}^2 = N_n \lr{F_n * f}(y).
  \]
  Thus, by \eqref{eq:var:p0}, we obtain
  \begin{equation}\label{eq:p1bound}
    1 - p_{1,n+m}(y)
    \ge \frac{\abs{p(y)}^2}{N_{n+m} \normltm{p}^2}
    = \frac{N_n}{N_{n+m}} \cdot \frac{f(y)}{\lr{F_n * f}(y)}
    \ge \lr{1 - \frac{d m}{n+m + 1}} \frac{1}{1 + h_n},
  \end{equation}
  where we define
  $h_n \coloneqq {\normlinftm{F_n * f - f}}/{f(y)}$,
  which proves the first statement.
  For the upper bound, we compute 
  \begin{align*}
    \abs{(F_n * f - f)(x)}
    &= \abs{\sum_{k\in\{-m,\dots,m\}^d} \sum_{\genfrac{}{}{0pt}{}{s\in\{0,1\}^d}{1\leq |s|\leq d}} \frac{(-1)^{|s|} |k^s|}{(n+1)^{|s|}} \hat{f}_k \eip{kx}} \\
    &= \abs{\sum_{k\in\{-m,\dots,m\}^d} \sum_{\genfrac{}{}{0pt}{}{s\in\{0,1\}^d}{1\leq |s|\leq d}} \frac{(-1)^{|s|}|k^s|}{(n+1)^{|s|}} \int_{\T^d} \abs{g(z)}^2 \eip{k(x-z)} \diff z} \\
    &\le \sum_{\genfrac{}{}{0pt}{}{s\in\{0,1\}^d}{1\leq |s|\leq d}} \left(\frac{m}{n+1}\right)^{|s|} (2m+1)^d \|g\|_{L^2}^2 \\
    &\leq  \|g\|_{L^2}^2 \frac{m(4m+2)^d}{n+1} 
  \end{align*}
  by using that $f=|g|^2$ is a trigonometric polynomial of degree $m$.
  Then it follows from \eqref{eq:p1bound} that
  \[
    p_{1,n+m}(y)
    \le h_n + \frac{d m}{\lr{n+m + 1}}
    \le \frac{\|g\|_{L^2}^2}{|g(y)|^2} \frac{m(4m+2)^d}{n+1} + \frac{d m}{\lr{n+m + 1}},
  \]
  since we can apply $(1+h_n)^{-1}\geq 1-h_n$.
\end{proof}

\section{Numerical examples}\label{sec:num}
We illustrate in this section the asymptotic behaviour of $p_n$ and $p_{1,n}$ for several types of singular measures, with respect to the Wasserstein-1 distance. We compute the distance using a semidiscrete optimal transport algorithm, described below. 

Our experiments focus on three examples on $\T^2$: a discrete measure $\mu_{\mathrm{d}}$ supported on $15$ points, with (nonnegative) random amplitudes, a uniform measure $\mu_{\mathrm{cu}}$ supported on the trigonometric algebraic curve
\begin{equation}\label{eq:implicit-curve}
    \cos(2\pi x)\cos(2\pi y) + \cos(2\pi x) + \cos(2\pi y) = \frac{1}{4},
\end{equation}
and a uniform measure $\mu_{\mathrm{ci}}$ supported on the circle centered in $c_0=(\frac{1}{2},\frac{1}{2})$ with radius $r_0 = 0.3$.

The moments of $\mu_{\mathrm{cu}}$ are computed numerically up to machine precision using Arb \cite{Johansson2017arb} with a parametrization of the implicit curve \eqref{eq:implicit-curve}. It follows from \eqref{eq:circlemoments} that the trigonometric moments of the measure $\mu_{\mathrm{ci}}$ are given by
\begin{equation*}
    \widehat{\mu_{\mathrm{ci}}}(k) = \eim{k c_0} J_0(2\pi r_0 \norm{k}_2).
\end{equation*}
The polynomials $p_n$, $J_n*\mu$, and $p_{1,n}$ can be evaluated efficiently via the fast Fourier transform over a regular grid in $\T^2$.
For the polynomial $p_{1,n}$, the singular value decomposition of the moment matrix $T_n$ can be computed at reduced cost by exploiting that $T_n$ has Toeplitz structure and resorting only to matrix-vector multiplications.

To compute transport distances to the measure $\mu\in\{\mu_{\mathrm{cu}},\mu_{\mathrm{ci}}\}$, let the curve $C=\supp\mu\subset \T^d$ denote its support with arc-length $L$. 
Now let $s\in\N$, take a partition $C=\bigcup_{\ell=1}^s C_\ell$ into path-connected curves with measure $\mu(C_\ell)=s^{-1}$ and arc-length $L_\ell$, and any $x_\ell\in C_\ell$, then
\begin{align*}
    W_1\left(\frac{1}{s}\sum_{\ell=1}^s \delta_{x_\ell}, \mu \right)
    &=\sup_{f: \Lip(f)\leq 1} \left|\sum_\ell \int_{C_\ell} \left[f(x)-f(x_\ell)\right] \diff \mu(x) \right|\\
    &\leq \sum_{\ell=1}^s \int_{C_\ell} |x-x_\ell|_1 \diff \mu(x) 
    \leq \sum_{\ell=1}^s \sqrt{d} L_\ell \mu(C_\ell) 
    =\frac{\sqrt{d}\cdot L}{s}.
\end{align*}
We denote the resulting discrete measures by $\mu_{\mathrm{cu}}^s$ and $\mu_{\mathrm{ci}}^s$, respectively (see Figure~\ref{fig:singular}). 
In our tests, we use $s = 3000$ samples, which offers a satisfactory tradeoff between computational time and accuracy for our range of degrees $n$. Indeed, the computational cost of evaluating the objective \eqref{eq:semidiscrete-objective} or its gradient grows linearly in $s$, while for degrees up to $n=250$, sampling beyond 3000 points has no effect on the output of our algorithm for computing $W_1(p_n,\mu^s)$, see Figure~\ref{fig:singular}.

\begin{figure}
    \captionsetup[subfloat]{labelformat=empty}
    \centering
    \subfloat[Algebraic curve]{\includegraphics[width=0.3\textwidth]{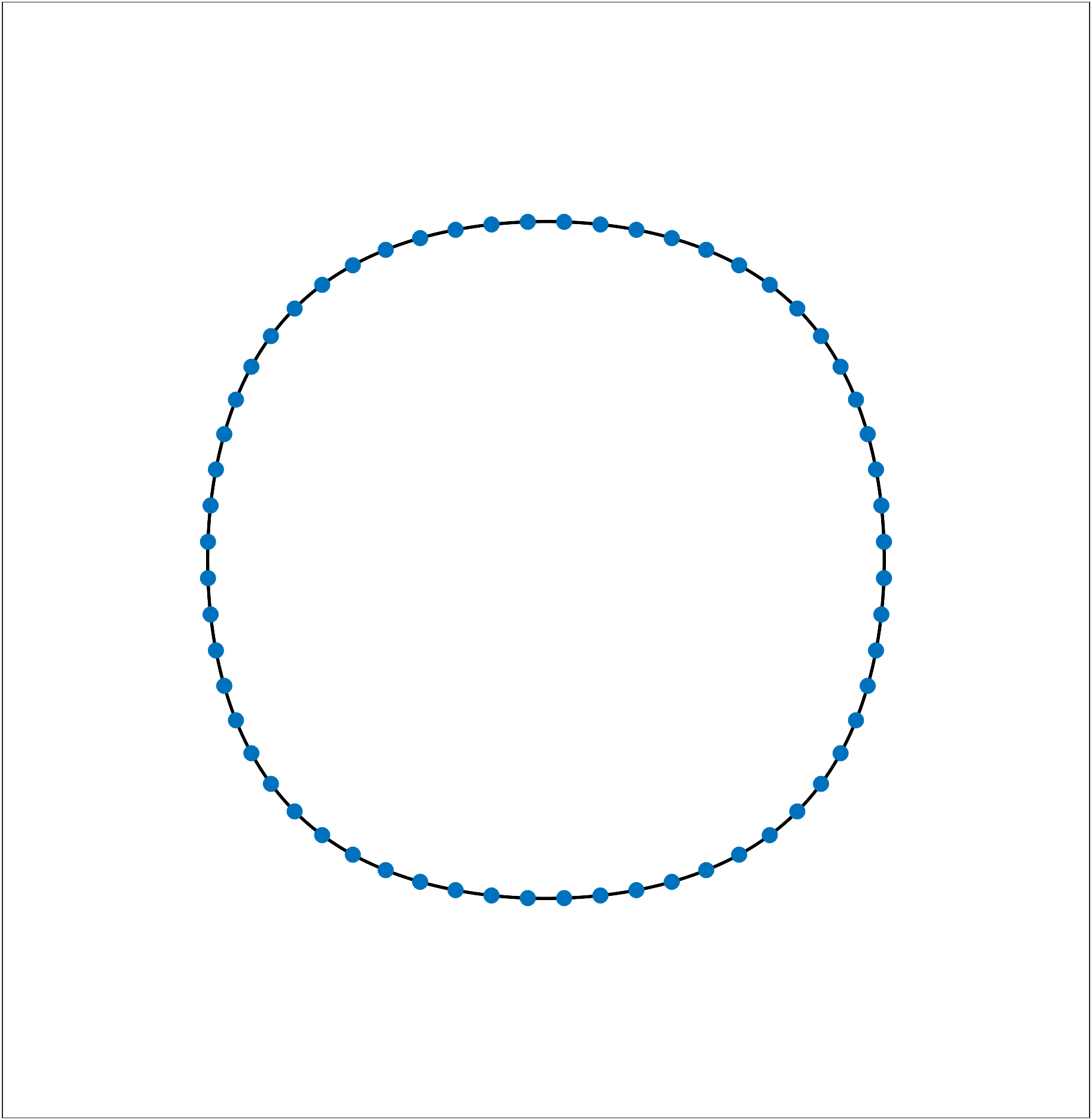}}
    \hspace{1mm}
    \subfloat[Circle]{\includegraphics[width=0.3\textwidth]{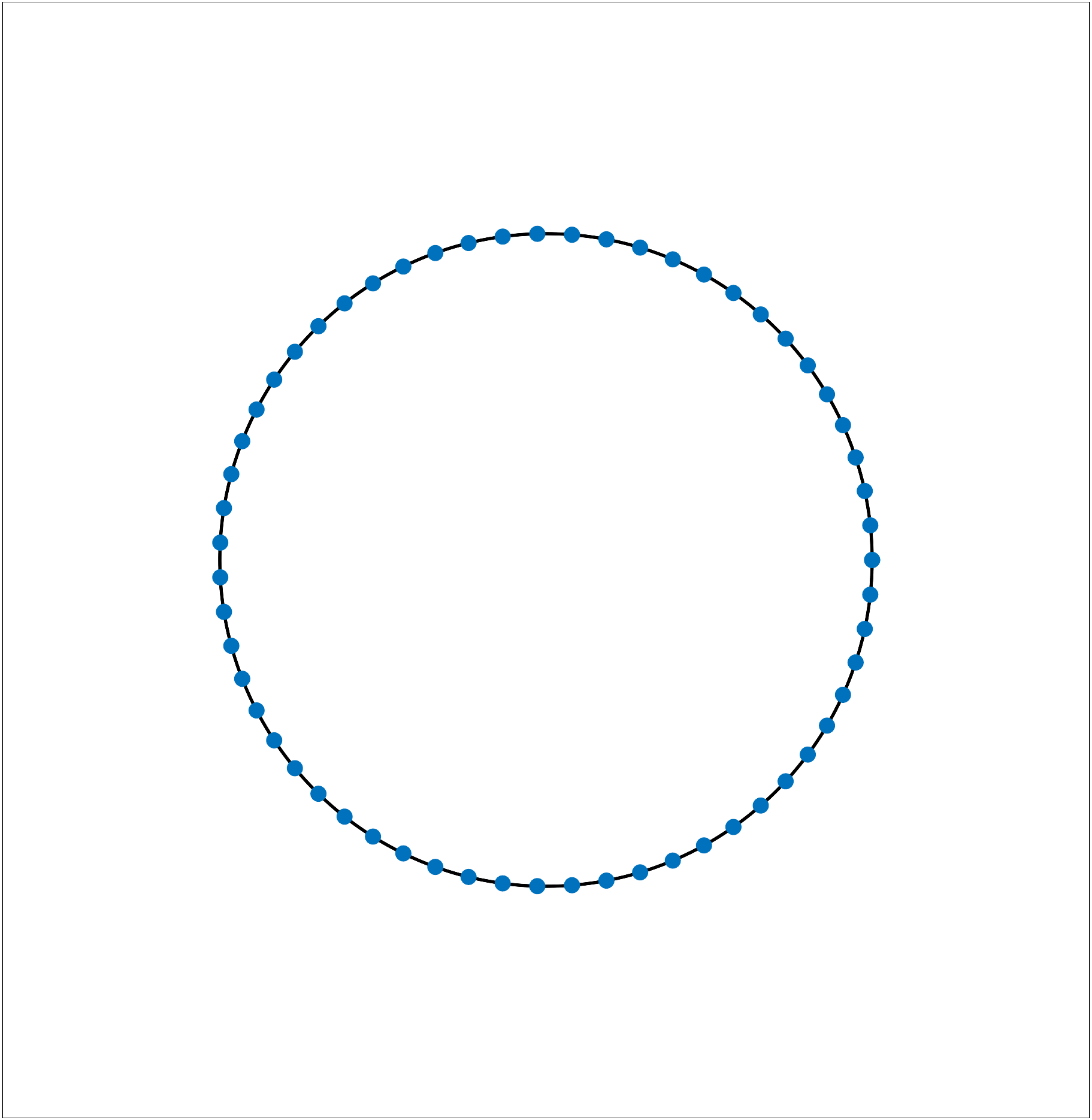}}
    \subfloat[Effect of sampling]{
        \includegraphics[width=0.3\textwidth,height=0.306\textwidth]{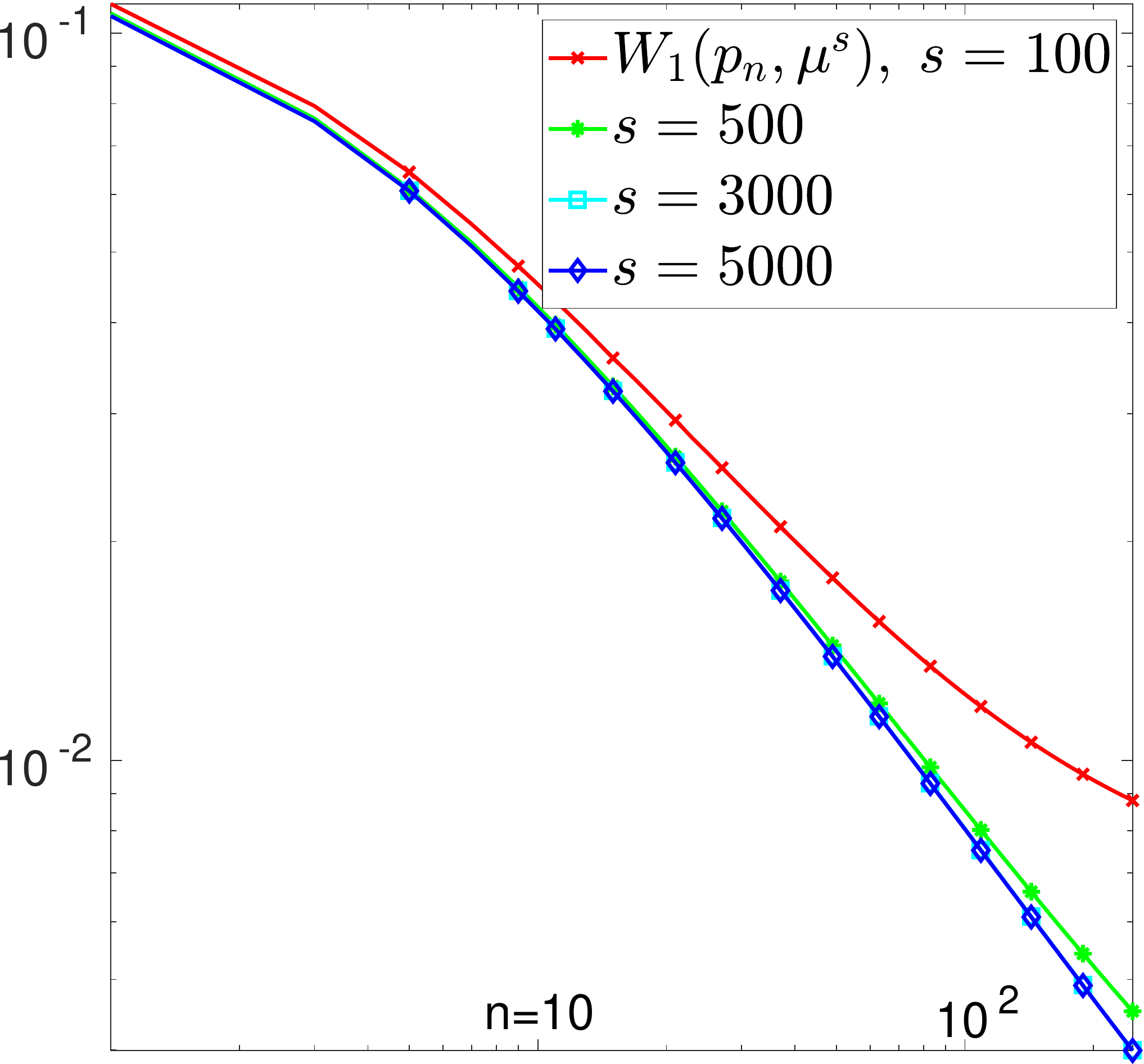}
    }
    \caption{The two example measures $\mu_{\mathrm{cu}}^s$ (left) and $\mu_{\mathrm{ci}}^s$ (middle) used in our numerical tests. In this display the two continuous measures are discretized using $s=60$ samples. The amplitudes of the spikes in both measures are taken equal, and normalized. The last plot shows the Wasserstein distance $W_1(F_n * \mu_{\mathrm{cu}},\mu_{\mathrm{cu}}^s)$ for degrees $n=1,\ldots,250$ and several values of $s$. }
    \label{fig:singular}
\end{figure}


Now let $\mu = \sum_{j=1}^s \la_j\de_{x_j}$ refer to either $\mu_{\mathrm{d}}$, $\mu_{\mathrm{cu}}^s$ or $\mu_{\mathrm{ci}}^s$. The semidiscrete optimal transport between a measure with density $p$ and the discrete measure $\mu$ may be computed by solving the finite-dimensional optimization problem
\begin{equation}
    \label{eq:semidiscrete-objective}
    \max_{w\in\R_+^s} f(w),\qquad f(w)=\sum_{j=1}^s \la_j w_j + \sum_{j=1}^s \int_{\Omega_j(w)} (\abs{x_j-y}-w_j)p(y)\diff y
\end{equation}
where the Laguerre cells associated to the weight vector $w$ are given by
\begin{equation*}
    \Omega_j(w) = \setcond{y\in\T^d}{\abs{x_j-y} - w_j \leq \abs{x_i-y}-w_i,\;i=1,\hdots,s},
\end{equation*}
see e.g.~\cite{Peyre_19}.
In our implementation, the density measure (and the Laguerre cells) are computed over a $502\times 502$ grid. We use a BFGS algorithm to perform the maximization, using the Matlab implementation \cite{Schmidt05}; we stop the iterations when the change of value of the objective goes below $10^{-9}$, or when the infinite norm $\norm{\nabla f}_\infty$ goes below $10^{-5}$. Note that this last condition has a geometrical interpretation since the $j$-th component of $\nabla f$ corresponds to the difference between the measure of the Laguerre cell $\Omega_j(w)$ and the amplitude $\la_j$. We set the limit number of iterations to 100. 

\begin{figure}[htb]
    \centering
    \captionsetup[subfloat]{labelformat=empty}
    \subfloat[Discrete]{\includegraphics[width=0.33\textwidth]{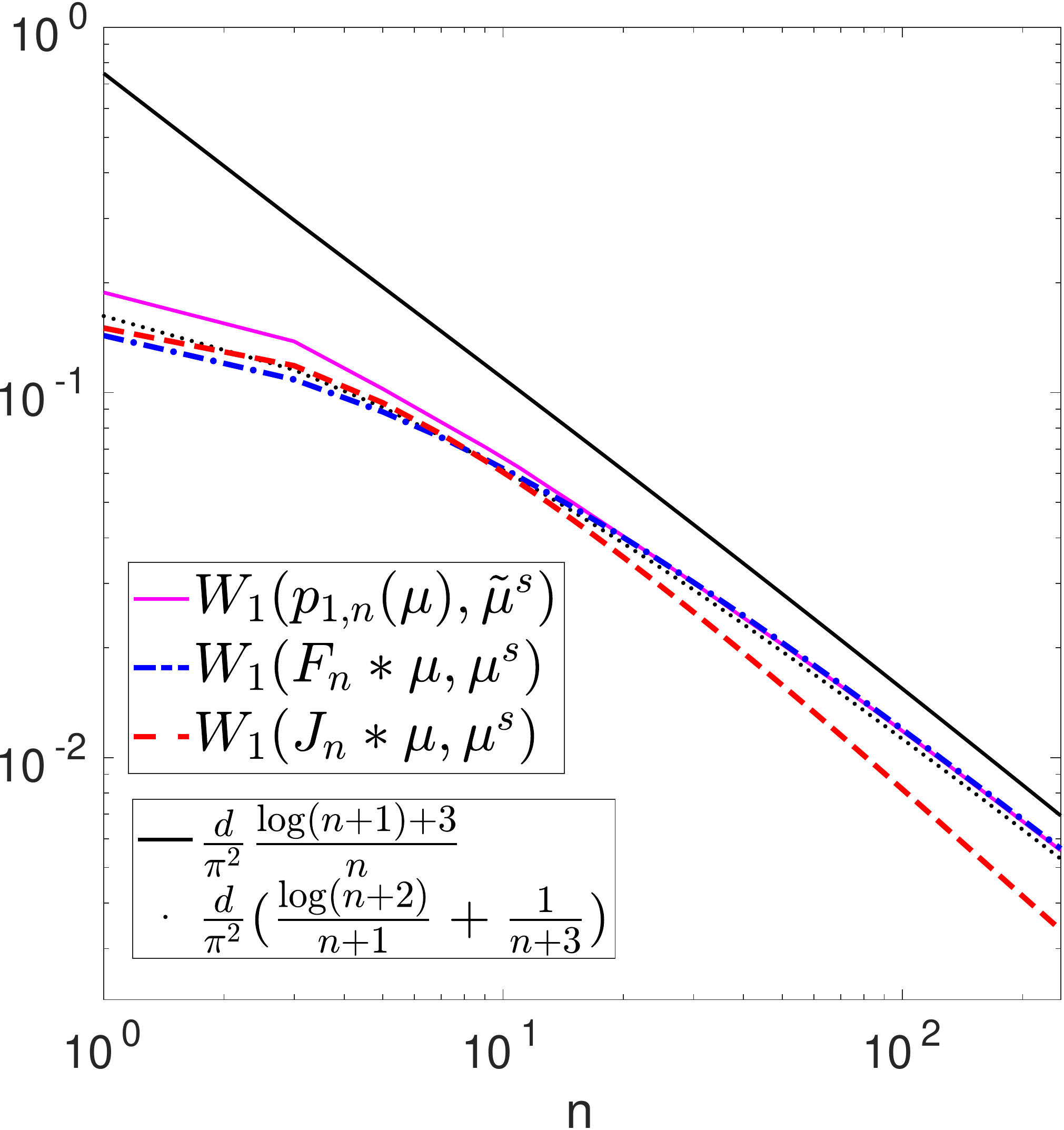}}
    \subfloat[Algebraic curve]{\includegraphics[width=0.33\textwidth]{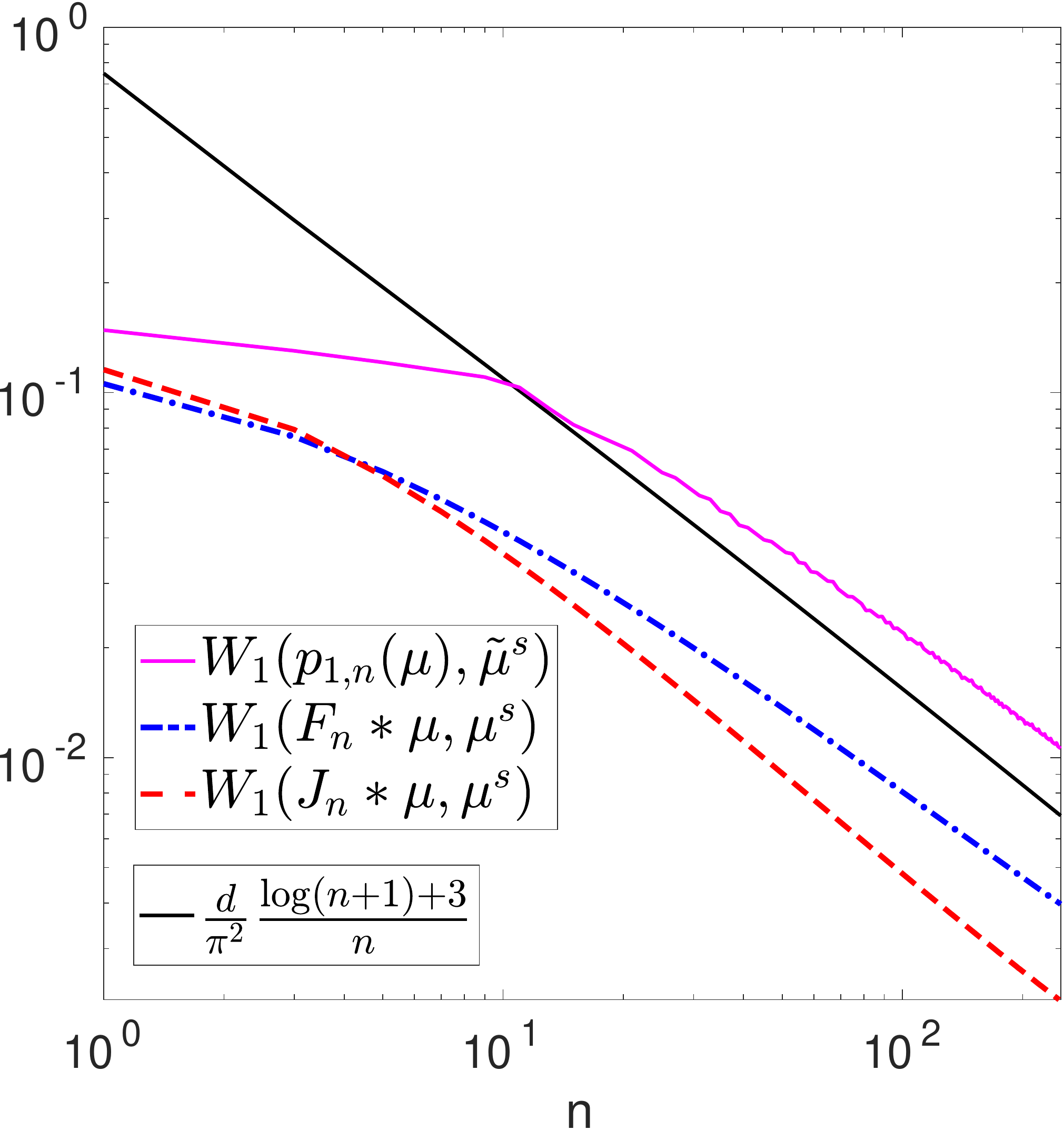}}
    \subfloat[Circle]{\includegraphics[width=0.33\textwidth]{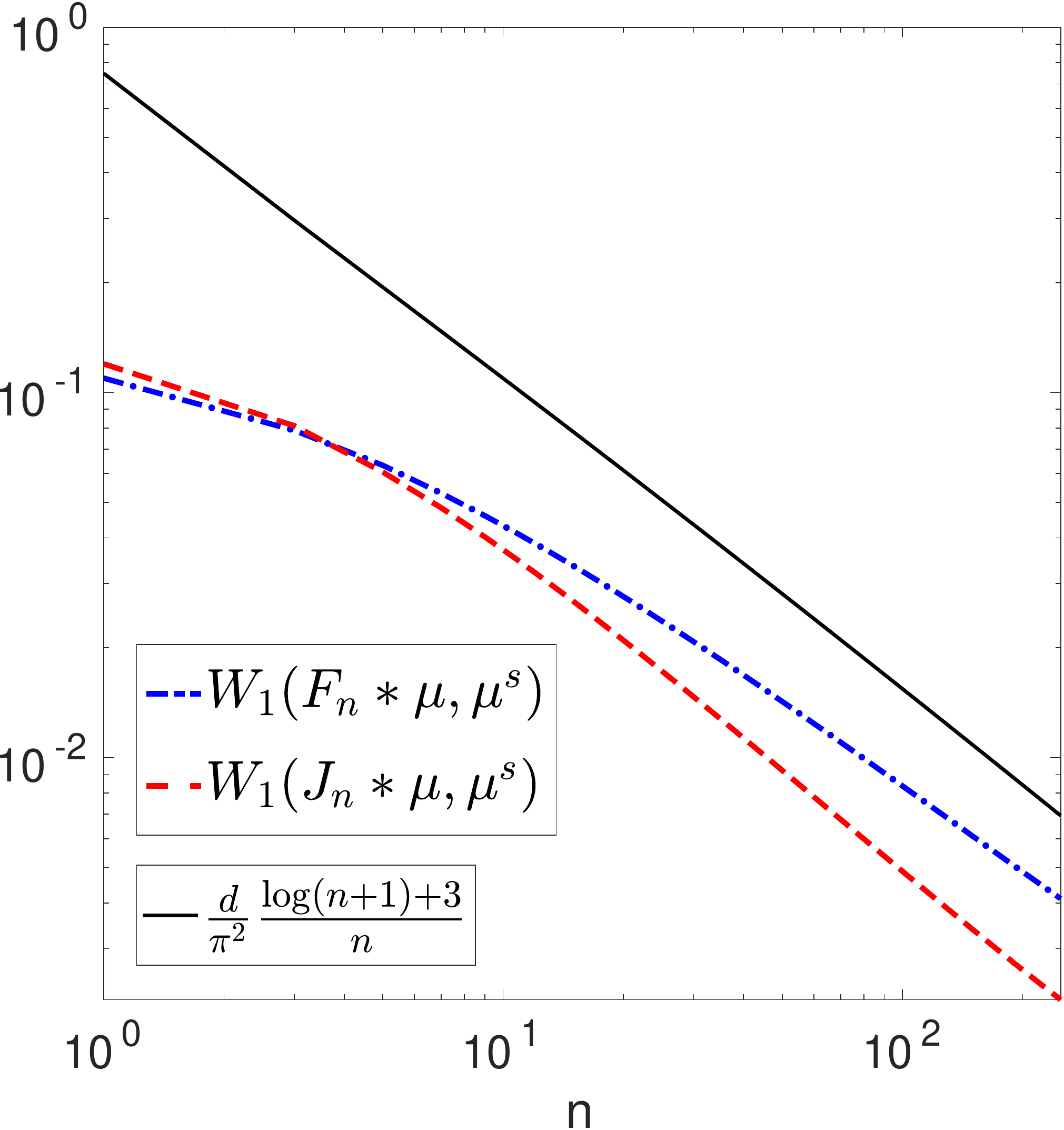}}
    \caption{Asymptotics of $p_n$ and $p_{1,n}$. For $p_{1,n}$, the distance is computed with respect to the unweighted measure $\tilde{\mu}^s$, that is $\tilde{\mu}^s = \frac{1}{s}\sum_{i=1}^s\de_{x_i}$ where $\{x_1,\dots,x_s\}$ is the support of $\mu$.}
    \label{fig:my_label}
\end{figure}

In the discrete case, our numerical results show that the Wasserstein distance $W_1(p_n,\mu^s)$ decreases at a rate close to the worst-case bound derived in Theorem~\ref{thm:W1p}. This is also the case for $W_1(p_{1,n},\tilde{\mu}^s)$, which is coherent with the bound given in the proof of Theorem~\ref{thm:p1weak}. In the positive dimensional cases, one would need to compute the Wasserstein distances for degrees larger than $n=250$ to be able to reliably estimate a rate, but this would require better optimized algorithms, in the spirit for instance of \cite{lakshmanan22}, which goes beyond the scope of this paper. Still, our preliminary results seem to indicate that the rates for $F_n * \mu$ and $J_n * \mu$ in the positive dimensional situation are similar to the ones for discrete measures, but with better constants. For $p_{1,n}$ on the other hand, although the theory does not foresee weak convergence in that case, if it were  to occur, our results indicate that the rate would then be worse than in the discrete case.

\section{Summary and outlook}
We provided tight bounds on the pointwise approximation error as well as with respect to the Wasserstein-1 distance when approximating arbitrary measures by trigonometric polynomials.

The approximation by the convolution with the Fejér kernel is both simple and up to a logarithmic factor best possible in the worst case. In contrast and beyond the scope of this paper, a computation similar to \cref{thm:W1p} and \cref{Rem_Jackson} shows that stronger localised trigonometric kernels seem necessary when considering the Wasserstein-2 distance.

Future work might also consider the truncation of the singular value decomposition in \cref{sec:interp} if the support of the measure is only approximated by the zero set of an unknown trigonometric polynomial or the available trigonometric moments are disturbed by noise.

\setlength{\emergencystretch}{1em}
\setcounter{biburlnumpenalty}{5000}
\printbibliography{}

\end{document}